\documentclass[12pt,reqno]{amsart}
\usepackage{color}
\usepackage[english]{babel}
\usepackage{amsfonts}
\usepackage{amsmath}
\usepackage{amsthm}
\usepackage{amssymb}
\usepackage[misc]{ifsym}
\usepackage{bbm}
\usepackage{mathrsfs}
\usepackage{esint}
\usepackage{faktor}
\usepackage[utf8]{inputenc}
\usepackage{afterpage}
\usepackage[left=2.9cm,right=2.9cm,top=2.8cm,bottom=2.8cm]{geometry}
\usepackage{graphicx}
\usepackage[pagewise,displaymath,mathlines]{lineno}
\usepackage{enumitem}
\usepackage[colorlinks=true,urlcolor=blue,citecolor=blue,linkcolor=blue,linktocpage,pdfpagelabels,bookmarksnumbered,bookmarksopen]{hyperref}

\newcommand{\N}{\mathbb{N}}

\newcommand{\R}{\mathbb{R}}

\newcommand{\Div}{\mathrm{div} \, }
\newcommand{\dx}{\, {\rm d} x}

\newcommand{\dtau}{\, {\rm d} \tau}
\newcommand{\Jac}{\, {\rm Jac}}
\newcommand{\eps}{\varepsilon}
\renewcommand{\phi}{\varphi}
\newcommand{\GEV}[1]{{\rm (}\hyperlink{prob}{{\rm GEV;} #1}{\rm )}}

\newtheorem{lemma}{Lemma}[section]
\newtheorem{thm}[lemma]{Theorem}
\newtheorem{prop}[lemma]{Proposition}

\theoremstyle{definition}
\newtheorem{defi}[lemma]{Definition}
\newtheorem{rmk}[lemma]{Remark}

\numberwithin{equation}{section}
\DeclareMathOperator*{\esssup}{ess \, sup}
\DeclareMathOperator*{\essinf}{ess \, inf}
\DeclareMathOperator*{\supp}{supp}

\usepackage{tikz}
\usetikzlibrary{patterns}
\usepackage{pgfplots}
\pgfplotsset{compat = newest}
\usepackage{caption}
\usepackage{subcaption}

\begin{document}
\title[Bobkov-Tanaka type spectrum]{On Bobkov-Tanaka type spectrum \\ for the double-phase operator}
	\author[L. Gambera]{Laura Gambera}
	\address[L. Gambera]{Dipartimento di Matematica e Informatica, Universit\`a degli Studi di Catania, Viale A. Doria 6, 95125 Catania, Italy}
	\email{laura.gambera@unipa.it}
\author[U. Guarnotta]{Umberto Guarnotta}
\address[U. Guarnotta]{Dipartimento di Ingegneria Industriale e Scienze Matematiche, Università Politecnica delle Marche, Via Brecce Bianche 12,
60131 Ancona, Italy}
\email{u.guarnotta@univpm.it}
	
	\begin{abstract}
	Moving from the seminal papers by Bobkov and Tanaka \cite{BT,BT2,BT3} on the spectrum of the $(p,q)$-Laplacian, we analyze the case of the double-phase operator. We discuss the region of parameters in which existence and non-existence of positive solutions occur. The proofs are based on normalization procedures, the Nehari manifold, and truncation techniques, exploiting Picone-type inequalities and an ad-hoc strong maximum principle.
	\end{abstract}

	\maketitle
	
	\let\thefootnote\relax
	\footnote{\Letter \quad Laura Gambera (laura.gambera@unipa.it).}

\section{Introduction and main result}

Let $\Omega\subseteq \R^N$, $N\ge 2,$ be a bounded domain with boundary $\partial \Omega$ of class $C^{1,\tau}$, $\tau\in(0,1]$, $\alpha, \beta \in \R,$ and $1<q<p<N,$ with $p<q^*:=\frac{Nq}{N-q}$. This paper concerns existence and non-existence of positive solutions to
	
	\begin{equation}
 \hypertarget{prob}{}
		\label{prob}
		\tag{GEV; $\alpha,\beta$}
		\left\{
		\begin{alignedat}{2}
			-\Delta_{p}^a u -\Delta_{q}u&= \alpha  a(x)|u|^{p-2}u+ \beta |u|^{q-2}u &&\quad \mbox{in}\;\; \Omega, \\
   u &=0 &&\quad \mbox{on}\;\; \partial\Omega. 
		\end{alignedat}
		\right.
	\end{equation} Here $\Delta_{p}^a$ is the weighted $p$-Laplacian, defined by  $$\Delta_{p}^a u= \Div (a(x) |\nabla u|^{p-2}\nabla u),$$ where $a\in C^{0,1}(\overline{\Omega})$ is positive in $\Omega$ and belongs to the Muckenhoupt class $A_p$, that is,
\begin{equation}
\label{Muck}
\frac{1}{|B|} \int_B a \dx \leq C \left(\frac{1}{|B|} \int_B a^{\frac{1}{1-p}} \dx \right)^{1-p}
\end{equation}
for some $C>0$ and any ball $B\subseteq \Omega$; see \cite[p.297]{HKM}.
  
  The differential operator in \eqref{prob} encompasses both the $(p,q)$-Laplacian and the double-phase operator. In particular, if $\inf_\Omega a >0,$ this operator has balanced $p$-growth, allowing the problem to be set in standard Sobolev spaces (see \cite{BT}). On the other hand, if $\inf_\Omega a=0$ then the operator has unbalanced growth, which makes \eqref{prob} fall into the Musielak-Orlicz setting.
  
  
  Double-phase operators were introduced in \cite{Z} to describe models of strongly anisotropic materials. In the same years, local regularity of minimizers of integrals with non-standard growth was investigated in \cite{PM1,PM2}. More recently, other local regularity results have been provided in \cite{CM,BCM,DFP,O,PM3, PM4} (see also the references therein). It is worth noticing that, in this setting, global regularity is far from being understood: indeed, the classical nonlinear regularity theory \cite{LI} is no more applicable for problems driven by unbalanced growth operators. Several existence, uniqueness, and multiplicity results for double-phase problems were obtained: here we mention \cite{LD,GP,GW1,GW2,GKS,PRZ,CGHW,CPW,GLW,CW,ABDW} and the survey \cite{P}.
  
  Since $\Delta_p^a+\Delta_q$ is not homogeneous, its spectrum can be defined in different ways. The first definition of spectrum has been introduced in \cite{CS} and is based on the Rayleigh quotient. A different type of spectrum was considered in \cite{GGP}: the eigenfunctions associated with the eigenvalue $\lambda>0$ are defined as solutions to
  \begin{equation*}
		\left\{
		\begin{alignedat}{2}
			-\Delta_p^a u-\Delta_q u &= \lambda a(x)|u|^{p-2}u \quad &&\mbox{in} \;\; \Omega, \\
			u &= 0 \quad &&\mbox{on} \;\; \partial\Omega.\\
		\end{alignedat}
		\right.
	\end{equation*}
 This definition can be generalized by taking into account the presence of the $q$-Laplacian. In this respect, a quite natural choice, reminiscent of the Fu\u{c}ik spectrum, is suggested by Bobkov and Tanaka (see, e.g., \cite{BT,BT2,BT3}), who investigated $(p,q)$-Laplacian problems. In particular, \cite{BT} concerns existence and non-existence of solutions to
  \begin{equation*}
		\label{prob1}
		\left\{
		\begin{alignedat}{2}
			-\Delta_{p}u -\Delta_{q}u&= \alpha  u^{p-1}+ \beta u^{q-1} &&\quad \mbox{in}\;\; \Omega, \\
            u &> 0 &&\quad \mbox{in} \;\; \Omega, \\
			u &=0 &&\quad \mbox{on}\;\; \partial\Omega.
		\end{alignedat}
		\right.
    	\end{equation*} Inspired by this paper, we address a similar issue in the setting of the double-phase operator.
     
     
     In order to state the main results of the paper, we introduce the eigenpair $(\lambda_1^a(p),\phi_p^a)$ and $(\lambda_1(q),\phi_q)$, related to $\Delta_p^a$ and $\Delta_q$, respectively (see Section 2 for details). The following linear independence condition will be pivotal in the description of the spectrum:
     \begin{equation}
         \label{LI}
         \tag{LI}
         \phi_p^{a}\neq k\phi_q \quad \text{for any }k\in\R.
     \end{equation}
For a discussion about \eqref{LI}, we address the reader to \cite[p.3280]{BT}, which focuses the $(p,q)$-Laplacian case. We also introduce the following constants, that will play a crucial role:
 \begin{equation}\label{smeno}\tilde{s}_{-}:=\frac{ \int_\Omega |\nabla \phi^{a}_{p}|^q \dx}{\int_\Omega (\phi^{a}_{p})^q \dx}, \quad \quad s^{*}_{-}:=\lambda_1^a(p)-\tilde{s}_{-},\end{equation}
\begin{equation}\label{spiù}
    \tilde{s}_{+}:= \frac{ \int_\Omega a|\nabla \phi_{q}|^p \dx}{\int_\Omega a\phi_{q}^p \dx}, \quad \quad s^{*}_{+}:=\tilde{s}_{+}-\lambda_1(q).
\end{equation}Let  \begin{equation}\label{sstar}
  s^*:= \lambda_1^a(p)-\lambda_1(q).\quad \quad 
\end{equation}
Due to the possible lack of regularity for $\phi^a_p$, $\tilde{s}_-$ and $s^*_-$ may be not well defined: thus, if $\int_\Omega |\nabla \phi^a_p|^q \dx$ is not finite, we posit $\tilde{s}_-:=+\infty$ and $s^*_-:=-\infty$. It is readily seen that $\tilde{s}_+\geq \lambda_1^a(p)$, $\tilde{s}_-\geq \lambda_1(q)$, and $s^*_-\le s^*\le s^*_{+}$, with strict inequalities if and only if \eqref{LI} holds true. We explicitly notice that $\tilde{s}_+$ and $s^*_+$ are well defined, since $\phi_q\in C^{1,\tau}(\overline{\Omega})$ (see Section 2).

	\makeatletter
	\tikzset{
		dot diameter/.store in=\dot@diameter,
		dot diameter=3pt,
		dot spacing/.store in=\dot@spacing,
		dot spacing=10pt,
		dots/.style={
			line width=\dot@diameter,
			line cap=round,
			dash pattern=on 0pt off \dot@spacing
		}
	}
	\makeatother

\begin{figure}[h]
\raggedright
\begin{subfigure}[c]{0.3\textwidth}
\begin{center}
\begin{tikzpicture}[every node/.style={scale=0.5}]


\begin{axis}[xmin = -2, xmax = 5, ymin=-2, ymax=5, axis lines = middle, xlabel=$\alpha$, ylabel=$\beta$, xtick={2,3}, ytick={1,4}, xticklabels={$ $,$\tilde{s}_+$}, yticklabels={$ $,$\tilde{s}_-$}]

\filldraw (2,0) circle (0pt) node[anchor=north west]{$\lambda_1^a(p)$};
\filldraw (0,1) circle (0pt) node[anchor=south east]{$\lambda_1(q)$};

\draw[thick] (3,1)--(5,1); 
\draw[thick] (2,4)--(2,5); 
\addplot[domain = 2:3,samples = 200,smooth,thick]{4/(3*x-5)}; 
\filldraw(2.1,2.2) circle (0pt) node{$\mathcal{C}$}; 
\filldraw (2,4) circle (1.5pt) node[anchor=south east]{$s=s^*_-$}; 
\filldraw (3,1) circle (1.5pt) node[anchor=north west]{$s=s^*_+$}; 
\filldraw (2,1) circle (1.5pt) node[anchor=south east]{$s=s^*$}; 

\addplot[domain = -2:5,samples = 200,smooth, dashed]{x+2}; 
\addplot[domain = -2:5,samples = 200,smooth, dashed]{x-1}; 
\addplot[domain = -2:5,samples = 200,smooth, dashed]{x-2}; 

\draw[dot diameter=1pt, dot spacing=3pt, dots] (-1.9,1)--(2,1); 
\draw[dot diameter=1pt, dot spacing=3pt, dots] (2,-1.9)--(2,1); 
\draw[fill=white,draw=none,pattern=dots] (-2,-2) rectangle (1.95,1); 
\path[pattern = dots] plot [domain = 2:3, samples = 200] (\x,{4/(3*\x-5)})
-- plot [domain = 3:2, samples = 200] (\x,{5}); 
\draw[fill=white,draw=none,pattern=dots] (3,1) rectangle (5,5); 
\end{axis}
\end{tikzpicture}
\caption{${\rm (LI)}$ holds true.}
\label{LItrue}
\end{center}
\end{subfigure}
\hspace{4cm}
\begin{subfigure}[c]{0.3\textwidth}
\begin{center}
\begin{tikzpicture}[every node/.style={scale=0.5}]


\begin{axis}[xmin = -2, xmax = 5, ymin=-2, ymax=5, axis lines = middle, xlabel=$\alpha$, ylabel=$\beta$, xtick={2}, ytick={1}, xticklabels={$ $}, yticklabels={$ $}]

\filldraw (2,0) circle (0pt) node[anchor=north west]{$\lambda_1^a(p)$};
\filldraw (0,1) circle (0pt) node[anchor=south east]{$\lambda_1(q)$};
\filldraw (2,1) circle (1.5pt) node[anchor=south east]{$s=s^*$}; 

\draw[dot diameter=1pt, dot spacing=3pt, dots] (-1.9,1)--(4.9,1); 
\draw[dot diameter=1pt, dot spacing=3pt, dots] (2,-1.9)--(2,4.9); 
\draw[fill=white,draw=none,pattern=dots] (-2,-2) rectangle (1.95,1); 
\draw[fill=white,draw=none,pattern=dots] (2.01,1) rectangle (4.95,4.95); 
\end{axis}
\end{tikzpicture}
\caption{${\rm (LI)}$ does not hold true.}
\label{LIfalse}
\end{center}
\end{subfigure}
\caption{\label{specfig}Summary of results. Existence of positive solutions is guaranteed in the white regions, while no positive solutions exist within the dotted regions. The curve $\mathcal{C}$ is represented with a black line, while the dotted oblique lines are related to the change of variable $s=\alpha-\beta$.}
\end{figure}
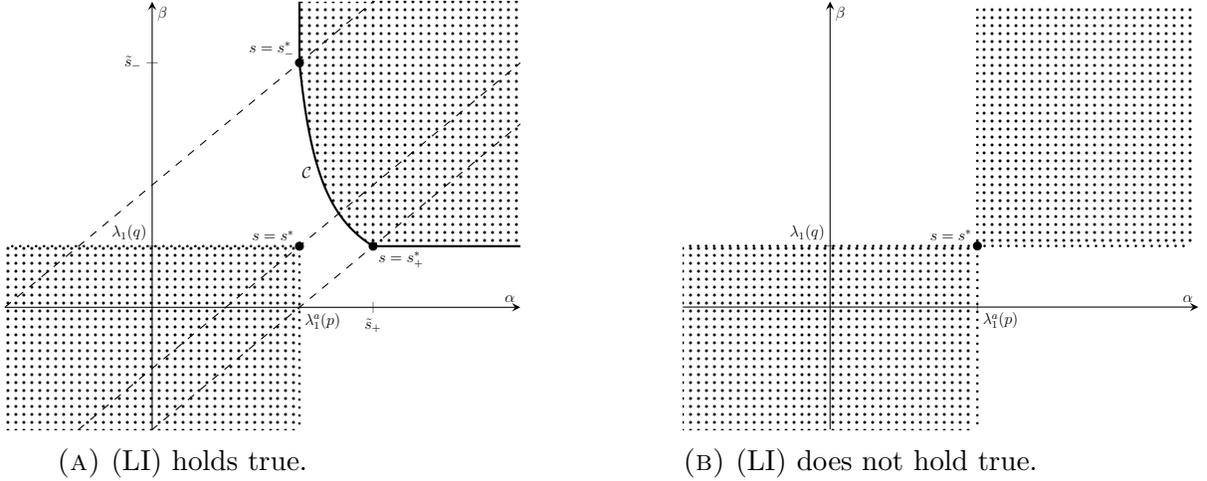

 Let us briefly sketch the main results of this paper. Regardless of \eqref{LI}, we have:
\begin{itemize}
\item non-existence for $(\alpha,\beta)\in\left((-\infty,\lambda_1^a(p)]\times(-\infty,\lambda_1(q)]\right) \setminus \{(\lambda_1^a(p),\lambda_1(q))\}$ (Theorem \ref{notex});
\item existence for $(\alpha,\beta)\in(-\infty,\lambda_1^a(p))\times(\lambda_1(q),+\infty)$ (Theorem \ref{normalize});
\item existence for $(\alpha,\beta)\in(\lambda_1^a(p),+\infty)\times(-\infty,\lambda_1(q))$ (Theorem \ref{nehari}).
\end{itemize}
If \eqref{LI} does not hold true, then we get:
\begin{itemize}
\item existence for $(\alpha,\beta)=(\lambda_1^a(p),\lambda_1(q))$ (Theorem \ref{notex});
\item non-existence for $(\alpha,\beta)\in(\lambda_1^a(p),+\infty)\times(\lambda_1(q),+\infty)$ (Proposition \ref{lambdaprop}).
\end{itemize}
On the other hand, if \eqref{LI} holds true, we consider the function $\lambda^*:\R\to\R$ defined as
\begin{equation}\label{definitionlambda}
\lambda^{*}(s):= \sup \{\lambda\in \R: (\hyperlink{prob}{\text{GEV}; \lambda+s, \lambda})\; \text{has a positive solution}  \}. \end{equation}
The curve $\mathcal{C}$ which separates the region of existence from the region of non-existence of positive solutions to \eqref{prob} corresponds to the set
$$\mathcal{C}= \{(\lambda^*(s)+s, \lambda^*(s)): s \in \R\}.$$
A description of $\mathcal{C}$ is given in Propositions \ref{epsilon}--\ref{lambdaprop}. Setting $s:=\alpha-\beta$, if \eqref{LI} is satisfied we obtain:
\begin{itemize}
\item non-existence for $(\alpha,\beta)=(\lambda_1^a(p),\lambda_1(q))$ (Theorem \ref{notex});
\item existence for $(\alpha,\beta)\in([\lambda_1^a(p),\tilde{s}_+)\times[\lambda_1(q),\lambda^*(s)))\setminus\{(\lambda_1^a(p),\lambda_1(q))\}$ (Theorems \ref{nehari2} and \ref{connessione}, besides Proposition \ref{epsilon});
\item non-existence for $(\alpha,\beta)\in(\lambda^*(s)+s,+\infty)\times (\lambda^*(s),+\infty)$ (obvious from the definition of $\lambda^*$);
\item existence for $(\alpha,\beta)=(\lambda^*(s)+s,\lambda^*(s))$ whenever $\alpha>\lambda_1^a(p)$ and $\beta>\lambda_1(q)$ (Theorem \ref{onthecurve}).
\end{itemize}

It is worth noticing that if $\inf_\Omega a > 0$ we have non-existence for $(\alpha,\beta)\in[\tilde{s}_+,+\infty)\times\{\lambda_1(q)\}$, while the validity of this non-existence result is an open problem when $\inf_\Omega a=0$; see Remarks \ref{openprob}-\ref{pqlaplacian}.

The curve $\mathcal{C}$ always touches the half-line $\mathcal{L}:=\{\lambda_1^a(p)\}\times(\lambda_1(q),+\infty)$, but the intersection point, depending on the relation between $p$ and $q$, may not belong to the line $\alpha-\beta=s^*_-$ (see \cite[Theorem 3.3]{BTPicone}); the truthfulness of this assertion is based on a specific Picone-type inequality (like \cite[Lemma 1]{I}). Incidentally, we point out a different Picone-type inequality obtained via hidden convexity in \cite{BF} (see \cite{M} for an application), that we will use in Lemma \ref{Picone2}.
 
\section{Preliminaries}
 Let $\R^N $ be the $N$-dimensional Euclidean space and $\Omega\subseteq\R^N$ be a bounded domain. With $B_r$ we indicate the generic ball of radius $r$; its measure will be denoted by $|B_r|$. We write $A\Subset \Omega$ to signify that the closure of the set $A\subseteq \R^N$ is contained in $\Omega$. \\
 Given $\gamma\in\R$, the symbol $\gamma_+:=\max\{\gamma,0\}$ stands for its positive part.

For any measurable $u:\Omega\to\R$ and $K\subseteq\Omega$, we denote by $\essinf_K u$ (resp., $\esssup_K u$) the essential infimum (resp., supremum) of $u$ on $K$. We recall that $u>0$ in $\Omega$ means $\essinf_K u>0$ for all $K\Subset \Omega$. \\
We denote by $ C^\infty_c(\Omega) $ the space of test functions that are compactly supported in $ \Omega $, while $C^{1,\tau}(\overline{\Omega})$, $\tau\in(0,1]$, stands for the space of continuously differentiable functions having $\tau$-H\"older-continuous gradient. \\
For any $1<p<\infty$, the symbol $\|\cdot\|_p$ indicates the usual norm of $L^p(\Omega)$, while $\|\cdot\|_{1,p}$ denotes the standard equivalent norm on $W^{1,p}_0(\Omega)$ stemming from Poincaré's inequality, that is,
\begin{equation*}
\| u\|_{1,p}:=\|\nabla u\|_p\, \quad \mbox{for all} \;\; u\in W^{1,p}_0(\Omega).
\end{equation*}

As said before, the unbalanced growth of the double-phase operator requires the usage of Musielak-Orlicz and Musielak-Sobolev-Orlicz spaces. For an exhaustive presentation of these spaces, we refer the reader to the monograph \cite{HH}; see also \cite{FZ,CGHW}.
\begin{defi}
A function $\phi: \Omega\times [0, +\infty)\to [0, +\infty) $ is called generalized $\Phi$-function if 
\begin{itemize}
    \item $\phi(\cdot,t)$ is measurable for all $t\in[0,+\infty)$;
    \item $\phi(x,\cdot)$ is a Young function for a.e.\,$x\in\Omega$ (see \cite[Definition 3.2.1]{KJF}). 
\end{itemize}
The set of all generalized $\Phi$-functions will be denoted by $\Phi(\Omega).$
\end{defi}
Let $\phi \in \Phi (\Omega)$ be such that $\phi(x,\cdot)$ satisfy the $\Delta_2$ condition for a.e.\,$x\in\Omega$ (see \cite[Definition 2.2.5]{HH}). The Musielak–Orlicz space $L^{\phi}(\Omega)$ is defined as
 \begin{equation*}
		L^{\phi}(\Omega):=\{u:\Omega\to \R: \; u\; \text{is measurable and }\, \rho_{\phi}(u) <\infty\}
	\end{equation*}
 through the modular function $$\rho_{\phi}(u)= \int_{\Omega}\phi(x, |u|)\dx.$$
$L^\phi(\Omega)$ is a Banach space when equipped with the Luxembourg norm
	\begin{equation*}
		\|u\|_ {\phi} := \inf \left\{ \lambda>0: \, \rho_{\phi}\left(\frac{u}{\lambda}\right) \leq 1 \right\}.
	\end{equation*}
The Musielak–Sobolev-Orlicz space  $W^{1,\phi}(\Omega)$ is defined by  
	\begin{equation*}
		W^{1,\phi}(\Omega):=\{u\in L^{\phi}(\Omega): |\nabla u|\in L^{\phi}(\Omega)\},
	\end{equation*}
	equipped with the norm
	\begin{equation*}
		\|u\|_ {W^{1,\phi}(\Omega)}:= \|u\|_{\phi}+ \|\nabla u\|_{\phi} .
	\end{equation*}
 We also introduce the space $W^{1,\phi}_{0}(\Omega)$ as the completion of $C^{\infty}_{c}(\Omega)$ under the norm $\|\cdot\|_{W^{1,\phi}(\Omega)}$. The topological dual of $W^{1,\phi}_0(\Omega)$ will be denoted with $W^{1,\phi}_0(\Omega)^*$, while $\langle \cdot,\cdot \rangle$ stand for the duality brackets. In the sequel, we will make use of $\theta_0, \theta\in\Phi(\Omega)$ defined by
 \begin{equation*}
  \begin{split} 
  \theta_0(x, t):= a(x)t^p \quad \text{and} \quad 
  \theta(x, t):= a(x)t^p+ t^q \quad \text{for all } (x,t) \in \Omega \times [0, +\infty),
  \end{split}
 \end{equation*}
being $1<q<p<N$ and $a: \Omega \to [0, +\infty)$ as in \eqref{prob}. Since Poincaré's inequality holds true for $W^{1,\theta}_0(\Omega)$ (see \cite[p.200]{CGHW}), then
$$ \|u\|_{1, \theta}= \|\nabla u\|_{\theta}\quad \text{for all}\; u \in W^{1,\theta}_{0}(\Omega) $$
is an equivalent norm in $W^{1,\theta}_0(\Omega)$. \\
We recall the following relation between norms and modulars in $L^\theta(\Omega)$ (see \cite[Lemma 3.2.9]{HH}):
	\begin{equation}
		\label{embrelation}
		\min \{\|u\|_{\theta}^{p}, \|u\|_{\theta}^{q} \}\le \rho_{\theta}(u)\le \max \{\|u\|_{\theta}^{p}, \|u\|_{\theta}^{q} \} \quad \text{for all}\; u \in L^{\theta}(\Omega).
	\end{equation}
  We conclude the presentation of the functional setting with the following result, which is a consequence of the bound $p<q^*$.
	%
 \begin{prop}
 \label{wemb}
 The embedding $W^{1,q}_0(\Omega)\hookrightarrow L^\theta(\Omega)$ is compact.
 \end{prop}
 \begin{proof}
 Take any bounded sequence $\{u_n\}\subseteq W^{1,q}_0(\Omega)$. Hence, according to Rellich-Kondrachov's theorem \cite[Theorem 9.16]{B}, there exists $u\in L^p(\Omega)$ such that $u_n \to u$ in $L^p(\Omega)$. Then the boundedness of $a$, ensuring that $L^p(\Omega)\hookrightarrow L^{\theta_0}(\Omega)$ continuously, and the continuity of the embedding $L^p(\Omega)\hookrightarrow L^q(\Omega)$ guarantee $u_n\to u$ in $L^\theta(\Omega)$.
 \end{proof}
 
Let $1<r<\infty$ and $w\in C^{0,1}(\overline{\Omega}) \cap A_p$ (see \eqref{Muck}) be a positive function. Consider the following weighted $r$-Laplacian eigenvalue problem:
\begin{equation}\label{proba}
		\left\{
		\begin{alignedat}{2}
			-\Delta_{r}^{w} v &= \lambda  w(x)|v|^{r-2}v &&\quad \mbox{in}\;\; \Omega,\\
   v &=0 &&\quad \mbox{on}\;\; \partial\Omega. 
		\end{alignedat}
		\right.
\end{equation}  
If $w\equiv 1,$ then \eqref{proba} reduces to the standard $r$-Laplacian eigenvalue problem; see \cite{Le}. When $w\equiv 1$, we will omit the superscripts $w$. The general case, encompassing $w=a$, was investigated in \cite{PPR}: in particular, \eqref{proba} admits a smallest eigenvalue $\lambda_{1}^w(r)$ which is positive, isolated, and simple, with corresponding positive eigenfunction $\phi_1^w(r)\in L^\infty(\Omega)$. These results heavily rely on the assumption $w\in A_p$, which makes Poincaré's inequality available: see \cite[Theorem 15.26, p. 307]{HKM} and \cite[Theorem 1]{HK}. If $\inf_\Omega w > 0$, then $ \phi_1^w(r)$ associated to $\lambda_1^w(r)$ belongs to $C^{1,\tau}(\overline{\Omega})$ for some $\tau\in(0,1]$; in this case we will assume $\|\phi_1^w(r)\|_{C^{1,\tau}(\overline{\Omega})}=1$. \\
Let us consider $\psi\in\Phi(\Omega)$ defined as
$$ \psi(x,t) := w(x)t^r \quad \text{for all } (x,t)\in\Omega\times[0,+\infty) $$
and the corresponding Musielak-Sobolev-Orlicz space $W^{1,\psi}(\Omega)$. Then the following variational characterization of the first eigenvalue holds true:
\begin{equation}\label{rayleigh}
    \lambda_1^{w}(r)= \inf_{v \in W^{1,\psi}_{0}(\Omega)\setminus\{0\} }\frac{\int_{\Omega}w|\nabla v|^r \dx}{\int_{\Omega}w|v|^r \dx}.
\end{equation}

The lack of $C^1$ regularity of solutions to \eqref{prob} prevents to use the classical strong maximum principle \cite[Theorem 1.1.1]{PS}; anyway, the weak Harnack inequality provided in \cite{BCM2} (see also \cite{BHHK}) allows us to recover this tool.
\begin{prop}
\label{strongmax}
Let $\Omega$ be a bounded domain. Suppose that $u\in W^{1,\theta}_0(\Omega)$ is a solution to
\begin{equation*}
\label{supersol}
-\Delta_p^a u - \Delta_q u \geq 0 \quad \text{in }\Omega.
\end{equation*}
Then either $u>0$ in $\Omega$ or $u\equiv 0$ in $\Omega$.
\end{prop}
\begin{proof}
Testing with $u_-$, one has $u\geq 0$ in $\Omega$. Assume $u\not\equiv 0$ in $\Omega$, so there exists $H\Subset \Omega$ such that $\esssup_H u > 0$. By contradiction, suppose that there exists $K\Subset \Omega$ such that $\essinf_K u=0$. Take any sub-domain $\Omega'\Subset \Omega$ such that $H\cup K \Subset \Omega'$. Set $\delta:={\rm dist}(\Omega',\partial\Omega)>0$. Then, fixed any $r<\min\{{\rm dist}(H\cup K,\partial\Omega'),\frac{\delta}{18}\}$, there exist $x_0\in K$, $x_1\in H$ such that $B_r(x_0)\cup B_r(x_1) \Subset \Omega'$ and
\begin{equation}
\label{balls}
\essinf_{B_r(x_0)} u \leq \essinf_K u = 0 < \esssup_H u \leq \esssup_{B_r(x_1)} u.
\end{equation}

First we construct a Harnack chain (see, e.g., \cite[pp.164-165]{PS}). To this aim, let us consider a continuous function $\Gamma:[0,1]\to\Omega'$ joining $\Gamma(0):=x_0$ with $\Gamma(1):=x_1$, whose existence is guaranteed since $\Omega'$ is open and connected. We consider the following open covering of $\Gamma([0,1])$:
$$ \mathcal{B}:=\left\{B(x,r): \, x\in\Gamma([0,1])\right\}. $$
Then the compactness of $\Gamma([0,1])$ allows to extract a finite subcovering 
$$\mathcal{B}':=\{B_1,B_2,\ldots,B_M\}.$$
We can assume $M\geq 2$, $B_1=B_r(x_0)$, and $B_M=B_r(x_1)$. Rearranging the elements of $\mathcal{B}'$, we can also suppose $B_i\cap B_{i+1}\neq \emptyset$ for all $i=1,\ldots,M-1$.

Now we prove that
\begin{equation}
\label{induction}
\esssup_{B_i} u = 0 \quad \text{for all }i=2,\ldots,M.
\end{equation}
Since $B_1=B_r(x_0)$ and $u\geq 0$ in $\Omega$, by \eqref{balls} we have $\essinf_{B_1} u = 0$. Now suppose that $\essinf_{B_i} u = 0$ for some $i\in\{1,\ldots,M-1\}$ and consider the ball $\hat{B}_i:=B_{4r}$ concentric with $B_i$. Since $B_i$ and $B_{i+1}$ are not disjoint and have the same radius, we get $B_{i+1}\subseteq \hat{B}_i$. According to \cite[Theorem 3.5 and Section 6]{BCM2} (see also \cite[Corollary 1.5 and Remark 1.6]{BHHK}), it turns out that
\begin{equation*}
0 = \essinf_{B_i} u \geq c\left(\int_{\hat{B}_i} u^l \dx\right)^{\frac{1}{l}} \geq c\left(\int_{B_{i+1}} u^l \dx\right)^{\frac{1}{l}} \geq 0
\end{equation*}
for opportune $c,l>0$. Hence $u\equiv 0$ a.e.\,in $B_{i+1}$, that is, $\esssup_{B_{i+1}} u = 0$. In particular, $\essinf_{B_{i+1}} u = 0$. Reasoning inductively yields \eqref{induction}.

The proof is concluded by noticing that $B_M=B_r(x_1)$ and \eqref{induction} for $i=M$ contradict \eqref{balls}.
\end{proof}

We conclude this section by introducing the variational setting of \eqref{prob}. The energy functional $E_{\alpha, \beta}: W^{1, \theta}_{0}(\Omega)\to \R$ associated  with \eqref{prob} is 
\begin{equation*}
	E_{\alpha, \beta}(u):= \frac{1}{p}H_{\alpha}(u)+\frac{1}{q}G_{\beta}(u),
\end{equation*}
being
\begin{equation*}\begin{split}
	&H_{\alpha}(u)= \int_{\Omega}a|\nabla u|^p \dx-\alpha\int_{\Omega}a| u|^p \dx \quad \text{and} \quad G_{\beta}(u)= \int_{\Omega}|\nabla u|^q \dx- \beta \int_{\Omega}| u|^q \dx\quad 
 \end{split}
\end{equation*}
for all $u \in W^{1,\theta}_0(\Omega)$. We note that $E_{\alpha, \beta}$ is well-defined and of class $C^1.$ The set
\begin{equation*}\label{Neharidef}
    \mathcal{N}_{\alpha, \beta}:=\{ u\in  W^{1,\theta}_{0}(\Omega)\setminus\{0\}: \;\langle E_{\alpha, \beta}'(u),u\rangle =0\},
\end{equation*}
is called Nehari manifold associated with $E_{\alpha, \beta}.$ We say that $u \in W^{1,\theta}_0(\Omega)$ is a ground-state solution to \eqref{prob} if it is a global minimizer of $E_{\alpha,\beta}\mid_{\mathcal{N}_{\alpha,\beta}}$, that is the restriction of $E_{\alpha,\beta}$ to its Nehari manifold. \\
The following result, patterned after \cite[Proposition 6]{BT}, furnishes a sufficient condition to ensure $\mathcal{N}_{\alpha, \beta}\not= \emptyset.$
\begin{prop}
\label{extremepoint}
Let $v\in W^{1,\theta}_0(\Omega)\setminus\{0\}$ be such that $$H_{\alpha}(v) G_{\beta}(v)<0.$$ Then there exist a unique critical point $t(v)>0$ of $t\mapsto E_{\alpha, \beta}(tv)$ such that $t(v)v \in \mathcal{N}_{\alpha, \beta}$. Moreover, if $$G_{\beta}(v)<0<H_{\alpha}(v),$$ then $t(v)$ is the unique minimizer of $t\mapsto E_{\alpha, \beta}(tv)$ and $E_{\alpha, \beta}(t(v)v)<0$.
\end{prop}

Hereafter, when no confusion arises, we will reason up to sub-sequences and zero-measure sets.

\section{existence and  non-existence results}
\begin{thm}\label{notex}
Let $$(\alpha, \beta)\in \left((-\infty, \lambda_1^{a}(p)]\times  (-\infty, \lambda_1(q)]\right)\setminus \{(\lambda_1^{a}(p),\lambda_1(q))\}.$$ Then \eqref{prob} admits no non-trivial solutions. Moreover, \eqref{LI} holds true if and only if \GEV{$\lambda_1^a(p),\lambda_1(q)$} admits no non-trivial (or, equivalently, positive) solutions.
		
		\begin{proof}
		Let $\alpha\le \lambda_1^{a}(p)$ and $\beta\le \lambda_1(q)$ fulfill $(\alpha,\beta)\neq(\lambda_1^{a}(p),\lambda_1(q))$. Suppose by contradiction that there exists $u\in W^{1,\theta}_0(\Omega)\setminus\{0\}$ solution to \eqref{prob}. Testing \eqref{prob} with $u$ yields \begin{equation*}
				\rho_{\theta_0}(\nabla u)-\alpha\rho_{\theta_0}(u)=\beta \|u\|^{q}_{q}-\|\nabla u\|^{q}_{q}.
	\end{equation*}
		According to \eqref{rayleigh} we get
		\begin{equation*}
		0\le ( \lambda_1^{a}(p)-\alpha)\rho_{\theta_0}(u)\le \rho_{\theta_0}(\nabla u)-\alpha\rho_{\theta_{0}}(u)=\beta\|u\|^{q}_{q}-\|\nabla u\|^{q}_{q}\le (\beta-\lambda_1(q))\|u\|^{q}_{q}\le 0.
		\end{equation*}
	 Since $u\neq 0,$ it follows that $\alpha= \lambda_1^{a}(p)$ and $\beta= \lambda_1(q)$, contradicting the hypotheses.
  
  Now suppose $(\alpha,\beta)=(\lambda_1^{a}(p),\lambda_1(q))$. Notice that, according to the simplicity of $\lambda_1^a(p)$ and $\lambda_1(q)$, the intersection of the eigenspaces associated to these eigenvalues is trivial if and only if \eqref{LI} holds true. \\
  Consequently, repeating the argument above, if \eqref{LI} holds true then \GEV{$\lambda_1^a(p),\lambda_1(q)$} admits only the trivial solution. On the other hand, if \eqref{LI} is not fulfilled, then there exists a positive $u\in W^{1,\theta}_0(\Omega)$ satisfying both $-\Delta_p^a u = \lambda_1^a(p)u^{p-1}$ and $-\Delta_q u = \lambda_1(q)u^{q-1}$, which entails
  $$-\Delta_p^a u - \Delta_q u = \lambda_1^a(p)u^{p-1} + \lambda_1(q)u^{q-1}.$$
    Thus, $u$ is a non-trivial solution to \GEV{$\lambda_1^a(p),\lambda_1(q)$}.
		\end{proof}
	\end{thm}

\begin{thm}
\label{normalize}
If \begin{equation*}
(\alpha, \beta)\in (-\infty, \lambda_1^{a}(p))\times  ( \lambda_1(q),+\infty),
\end{equation*} then there exists a positive solution $u\in W^{1,\theta}_0(\Omega)$ to \eqref{prob}, which is a global minimizer of $E_{\alpha,\beta}$.
\end{thm}
\begin{proof}
For any $t>0$ sufficiently small one has 
\begin{equation*} \begin{split}
  E_{\alpha, \beta}(t\phi_q)&=\frac{t^p}{p}\left(\rho_{\theta_0} (\nabla\phi_q)-\alpha \rho_{\theta_0} (\phi_q)\right)+\frac{t^q}{q}\left( \|\nabla\phi_q\|^{q}_q-\beta\|\phi_q\|^{q}_{ q}\right)\\ 
&= \frac{ t^p}{p}\left( \rho_{\theta_0} (\nabla\phi_q)-\alpha \rho_{\theta_0} (\phi_q)\right)+\frac{t^q}{q}\left( \lambda_1(q)-\beta\right)\|\phi_q\|^{q}_{ q}< 0.
\end{split}
\end{equation*} Let  $m:=\inf_{W^{1,\theta}_0(\Omega)} E_{\alpha, \beta}\in [-\infty,0)$ and $\{u_n\}\subseteq W^{1, \theta}_{0} (\Omega) $ be a minimizing sequence for $E_{\alpha,\beta}$. Since $m<0,$ we can assume that $E_{\alpha, \beta}(u_n)<0$ for all $n\in\N$.

\underline{Claim}: $\{u_n\}$ is bounded in $W^{1, \theta}_{0} (\Omega)$. \\
Suppose by contradiction that $\|u_n\|_{1,\theta}\to\infty$ and set $v_n:= \frac{u_n}{\|u_n\|_{1,\theta}}$ for all $n \in \N$. The boundedness of $\{v_n\}$ in $W^{1, \theta}_{0} (\Omega)$ and the compactness of $W^{1,\theta}_0(\Omega)\hookrightarrow L^\theta(\Omega)$ imply the existence of $v\in W^{1, \theta}_{0}(\Omega)$ such that  \begin{equation}\label{weaklimit}
    v_n \rightharpoonup v \quad \text{in } W^{1, \theta}_{0} (\Omega) \quad\text{and}\quad v_n \to v\quad \text{in } L^{\theta}(\Omega).
\end{equation}
By \eqref{rayleigh} we have, for all $n\in\N$,
\begin{equation}\label{limit}
\begin{split}
0&>  E_{\alpha, \beta}(u_n)=   E_{\alpha, \beta}(\|u_n\|_{1,\theta} v_n)\\&= \frac{1}{p}\|u_n\|_{1,\theta}^p(\rho_{\theta_0}(\nabla v_n)- \alpha \rho_{\theta_0}( v_n))+\frac{1}{q}\|u_n\|_{1,\theta}^q(\|\nabla v_n\|_q^{q} -\beta \|v_n\|^{q}_{q})\\
&\ge \frac{1}{p}\left(1-\frac{\alpha_+}{\lambda^a_1(p)} \right)\rho_{\theta_0}(\nabla v_n)\|u_n\|_{1,\theta}^p +\frac{1}{q}\|u_n\|_{1,\theta}^q(\|\nabla v_n\|_q^{q} -\beta \|v_n\|^{q}_{q}).
\end{split}
\end{equation}
Exploiting \eqref{rayleigh} again, besides recalling that $\beta>\lambda_1(q)$ and $\|\nabla v_n\|_q^q\leq \rho_\theta(\nabla v_n)=1$ (see \eqref{embrelation}), one has
$$\|\nabla v_n\|_q^{q} -\beta \|v_n\|^{q}_{q} \geq -\left(\frac{\beta}{\lambda_1(q)}-1\right)\|\nabla v_n\|_q^q \geq 1-\frac{\beta}{\lambda_1(q)} \quad \text{for all }n\in\N.$$
Hence, if 
\begin{equation}\label{liminfnot0}
    \limsup_{n \to \infty}\rho_{\theta_0}(\nabla v_n)>0,
\end{equation} 
then \eqref{limit} yields (up to sub-sequences)
$$ 0 > \eps\|u_n\|_{1,\theta}^p+\frac{1}{q}\left(1-\frac{\beta}{\lambda_1(q)}\right)\|u_n\|_{1,\theta}^q \quad \text{for all }n\in\N, $$
being $\eps>0$ sufficiently small. Letting $n\to\infty$ leads to a contradiction. On the other hand, if \eqref{liminfnot0} is not satisfied, then $v_n\to 0$ in $W^{1, \theta_0}_{0}(\Omega)$. This implies, via \eqref{embrelation} (which guarantees that $\rho_\theta(\nabla v_n)=1$ if and only if $\|\nabla v_n\|_\theta=1$), that $\|\nabla v_n\|_q^q=1-\rho_{\theta_0}(\nabla v_n)\to 1$ as $n\to\infty$. Using also \eqref{weaklimit} one has $\|v_n\|_q^q\to 0$ when $n\to\infty$. Hence, for all $n\in\N$ large enough, one has $\|\nabla v_n\|_q^q-\beta\|v_n\|_q^q\geq\frac{1}{2}$, so that \eqref{limit} becomes
$$ 0 > \frac{1}{q}\|u_n\|_{1,\theta}^q(\|\nabla v_n\|_q^{q} -\beta \|v_n\|^{q}_{q}) \geq \frac{1}{2q}\|u_n\|_{1,\theta}^q, $$
which is a contradiction establishing the claim.

Boundedness of $\{u_n \}$ in $W^{1,\theta}_{0}(\Omega)$ ensures the existence of $u \in W^{1,\theta}_{0}(\Omega)$ such that $$u_n \rightharpoonup u \quad \text{in } W^{1,\theta}_{0}(\Omega)\quad \text{and}\quad u_n \rightarrow u \quad \text{in }  L^{\theta}(\Omega).$$ 
By the weak lower semi-continuity of $E_{\alpha,\beta}$, we have $$m=\lim_{n \to \infty} E_{\alpha, \beta}(u_n)\ge E_{\alpha, \beta}(u),$$
whence $m>-\infty$, while the minimality of $m$ entails $m=E_{\alpha, \beta}(u).$ In particular, $u$ is a solution to \eqref{prob}. Since $E_{\alpha, \beta}$ is even, we can assume $u\ge0$ in $\Omega$.  Moreover, since $m<0=E_{\alpha, \beta}(0)$, we infer  $u\neq 0$. Finally, Proposition \ref{strongmax} guarantees that $u>0$ in $\Omega$.
\end{proof}

\begin{thm}
\label{nehari}
Suppose \begin{equation*}
(\alpha, \beta)\in ( \lambda_1^{a}(p),+\infty)\times(-\infty, \lambda_1(q)).
\end{equation*} Then there exists a positive solution $u\in W^{1,\theta}_0(\Omega)$ to \eqref{prob}, which is a ground-state solution with positive energy.
\end{thm}
\begin{proof}
Firstly, we notice that Proposition \ref{extremepoint} ensures $\mathcal{N}_{\alpha, \beta}\not=\emptyset$.

\underline{Claim 1}: $E_{\alpha, \beta}\mid_{\mathcal N_{\alpha, \beta}}$ is coercive. \\
Reasoning as in \cite[Proposition 3.3]{GGP}, we suppose by contradiction that there exist $M>0$ and $\{u_n\}\subseteq \mathcal{N}_{\alpha, \beta}$ such that $\|u_n\|_{{1,\theta}} \to \infty$ and 
\begin{equation} \label{absurd}
\begin{split}
	  E_{\alpha, \beta}(u_n)\le M\quad \text{for all } n \in \N.
\end{split}\end{equation} 
Fix any $n\in\N$. Since $u_n\in\mathcal{N}_{\alpha, \beta}$, we have 
\begin{equation}\label{neharibelong}
\rho_{\theta_0}(\nabla u_n) + \|\nabla u_n\|_q^q  = \alpha \rho_{\theta_0}(u_n) + \beta\|u_n\|_q^q,
\end{equation}
so \eqref{absurd} can be rewritten as 
\begin{equation*}
\label{coercivity1}
\begin{split}
E_{\alpha, \beta}(u_n)=\left(\frac{1}{q}-\frac{1}{p} \right)(\|\nabla u_n\|_q^q-\beta\|u_n\|_q^q)\le M.
\end{split}
\end{equation*} Set $y_n:= \frac{u_n}{\|u_n\|_{\theta}}$. Recalling \eqref{rayleigh}, besides $ \beta<\lambda_1(q)$ and $q<p$, one has  
\begin{equation}\label{betacond}
0\le \frac{E_{\alpha, \beta}(u_n)}{\|u_n\|_\theta^q}=\left(\frac{1}{q}-\frac{1}{p} \right)(\|\nabla y_n\|_q^q-\beta\|y_n\|_q^q)\le\frac{M}{\|u_n\|_{\theta}^{q}}.
	\end{equation}
 Moreover, \eqref{neharibelong} ensures
 \begin{equation*}\label{nablatou}
     \rho_{\theta}(\nabla u_n)\le \max\{\alpha,\beta\} \rho_{\theta}(u_n) \quad \text{for all }n\in\N.
 \end{equation*}
 Hence, by Poincaré's inequality, $\|u_n\|_{1,\theta} \to \infty$ forces $\|u_n\|_{\theta} \to \infty.$ Thus, passing to the limit in \eqref{betacond} we get
\begin{equation} \label{qconvergece0}
\lim_{n\to\infty} (\|\nabla y_n\|_q^q-\beta\|y_n\|_q^q) = 0.
\end{equation}
Exploiting also \eqref{rayleigh}, along with $\beta<\lambda_1(q)$, one has 
\begin{equation}\label{w1q}
    0\le \left(1-\frac{\beta_+}{\lambda_1(q)}\right)\|\nabla y_n\|_q^q  \le \|\nabla y_n\|_q^q - \beta\|y_n\|_q^q \to 0 \quad \text{as }n\to\infty,
\end{equation}
 whence $ y_n\to 0$ in $W^{1,q}_{0}(\Omega).$  
Finally, according to Proposition \ref{wemb}, it turns out that $y_n \to 0$ in $ L^{\theta}(\Omega),$ contradicting $\|y_n\|_{\theta}=1$ for all $n\in \N$. We deduce the coercivity of $E_{\alpha, \beta}\mid_{\mathcal N_{\alpha, \beta}}$.

\underline{Claim 2}: $\displaystyle{m_{\alpha,\beta}:=\inf_{ \mathcal{N}_{\alpha, \beta}}E_{\alpha, \beta}}>0$. \\
Since $q<p$ and $\beta<\lambda_1(q),$ by \eqref{neharibelong} we have 
\begin{equation*}
\begin{split}
E_{\alpha, \beta}(u)=\left(\frac{1}{q}-\frac{1}{p} \right)(\|\nabla u\|_q^q -\beta\|u\|_q^q)>0 \quad \text{for all } u \in \mathcal{N}_{\alpha, \beta}.
\end{split}
\end{equation*}
Hence $m_{\alpha, \beta} \ge0$. To show that $m_{\alpha,\beta}>0$, we argue by contradiction. Let $\{u_n\}\subseteq \mathcal{N}_{\alpha, \beta}$ be such that $E_{\alpha,\beta}(u_n)\to 0$ when $n\to\infty$. Thus, reasoning as in \eqref{w1q}, we get $u_n \to 0$ in $W^{1,q}_{0}(\Omega).$ Moreover, the coercivity of  $E_{\alpha, \beta}\mid_{\mathcal N_{\alpha, \beta}}$ guarantees that $\{u_n\}$ is bounded in $W^{1, \theta}_{0}(\Omega)$, which compactly embeds into $L^\theta(\Omega)$. Hence
\begin{equation*}
    u_n\rightharpoonup 0 \quad \text{in } W^{1, \theta}_{0}(\Omega) \quad \text{and}\quad  u_n\rightarrow 0 \quad \text{in } L^{ \theta}(\Omega).
\end{equation*} 
Set $y_n:= \frac{u_n}{\|u_n\|_{\theta}}$ for all $n\in \N.$  Using $\beta<\lambda_1(q)$, \eqref{neharibelong}, \eqref{rayleigh}, and $\rho_{\theta_0} (y_n)\le 1 $ for all $n \in \N,$ we get  \begin{equation}
\label{yncontrol}
\begin{split}
	 0&\le\|\nabla y_n\|_q^q - \beta\|y_n\|_q^q =\| u_n\|_{\theta}^{p-q}\big[ \alpha \rho_{\theta_0}(y_n) - \rho_{\theta_0}(\nabla y_n)\big] \\
	& \le\| u_n\|_{\theta}^{p-q}  \left( \alpha  - \lambda_1^{a}(p)\right) \rho_{\theta_0}( y_n) \le  \| u_n\|_{\theta}^{p-q}\left( \alpha  - \lambda_1^{a}(p)\right).
\end{split}
\end{equation} Hence, recalling that  $u_n\to 0$ in $L^{\theta}(\Omega)$, $p>q$, and $\alpha>\lambda_1^a(p)$, it turns out that $$y_n\to 0\quad \text{in}\quad W^{1,q}_{0}(\Omega).$$ Owing to Proposition \ref{wemb} we infer $$y_n \to 0 \quad \text{in } L^{\theta}(\Omega),$$
contradicting $\|y_n\|_{\theta}=1$ for all $n \in \N$. Accordingly, $m_{\alpha,\beta}>0$.

Reasoning as in \cite[Proposition 3.5 and Theorem 3.6]{GGP} furnishes $u\in \mathcal{N}_{\alpha, \beta}$ such that $E_{\alpha, \beta}(u)=m_{\alpha,\beta}$ and $u$ is a (unconstrained) critical point of $E_{\alpha,\beta}$. Then the conclusion follows as in Theorem \ref{normalize}.
\end{proof}

\begin{thm}
\label{nehari2}
Assume that \eqref{LI} holds true. If \begin{equation*}\begin{split}
	      (\alpha, \beta)\in (\lambda_1^a(p), \tilde{s}_+)\times \{\lambda_1(q)\},
	 \end{split}
\end{equation*} then there exist a positive solution $u\in W^{1,\theta}_0(\Omega)$ to \GEV{$\alpha, \lambda_1(q)$}, which is a ground-state solution with positive energy.
\end{thm}
\begin{proof}
The proof follows the argument used in Theorem \ref{nehari}.

\underline{Claim 1}: $\mathcal{N}_{\alpha, \lambda_1(q)}\not=\emptyset.$ \\
Recalling that $\alpha>\lambda_1^{a}(p)$ and $W^{1,\theta}_0(\Omega)$ is densely embedded into $W^{1,\theta_0}_0(\Omega)$ (see \cite[Theorem 2.23]{CGHW}), there exists $\bar{u}\in W^{1, \theta}_{0} (\Omega)\setminus \{0\}$  such that 
 \begin{equation}\label{notnehari}
     \alpha \rho_{\theta_0}(\bar{u}) > \rho_{\theta_0}(\nabla \bar{u}).
 \end{equation}
 Since $\alpha<\tilde{s} _{+}$, \eqref{notnehari} implies $\bar{u}\neq t \phi_q$ for any $t\in \R.$ 
 Thus,
 \begin{equation}
    \label{notnehari2}
    \|\nabla \bar{u}\|_q^q>\lambda_1(q)\|\bar{u}\|_q^q.
 \end{equation}
By \eqref{notnehari}--\eqref{notnehari2} we deduce $H_\alpha(\bar{u})<0<G_{\lambda_1(q)}(\bar{u})$, so Proposition \ref{extremepoint} proves the claim.

\underline{Claim 2}: $E_{\alpha, \lambda_1(q)}|_{\mathcal N_{\alpha, \lambda_1(q)}}$ is coercive.\\
Suppose by contradiction that there exist $M>0$ and $\{u_n\}\subseteq \mathcal{N}_{\alpha, \lambda_1(q)}$ such that $\|u_n\|_{1,\theta}\to\infty$ and
\begin{equation}
    \label{boundcoerc}
    E_{\alpha, \lambda_1(q)}(u_n)\le M \quad \text{for all }n\in\N.
\end{equation}
Set $y_n:= \frac{u_n}{\|u_n\|_{\theta}}$ for all $n\in\N$. Since $\{y_n\}$ is bounded in $L^\theta(\Omega)$, then $y_n\rightharpoonup y$ in $L^\theta(\Omega)$ for some $y\in L^\theta(\Omega)$. Reasoning as for \eqref{qconvergece0} we get
\begin{equation}
\label{Wbound1}
\lim_{n\to\infty} (\|\nabla y_n\|_q^q - \lambda_1(q) \|y_n\|_q^q) = 0.
\end{equation}
In particular, recalling that $\|y_n\|_q\leq \|y_n\|_\theta = 1$ for all $n\in\N$, \eqref{Wbound1} ensures that $\{y_n\}$ is bounded in $W^{1,q}_0(\Omega)$, whence $y_n\rightharpoonup y$ in $W^{1,q}_0(\Omega)$. Then the weak lower semi-continuity of $\|\cdot\|_q$ ensures
\begin{equation}
\label{Wbound2}
\|\nabla y\|_q^q - \lambda_1(q) \|y\|_q^q=0.
\end{equation}
Hence, only two cases can occur: \begin{itemize}
    \item $y_n \rightharpoonup 0$ in $W^{1,q}_{0}(\Omega),$ or
    \item $y_n \rightharpoonup t\phi_q$ in $W^{1,q}_{0}(\Omega)$ for some $t \in \R\setminus\{0\}$.
\end{itemize}
In the first case, Proposition \ref{wemb} forces $y_n\to 0$ in $L^{\theta}(\Omega)$, which contradicts   $\|y_n\|_{\theta}=1$ for all $n\in\N$. \\ 
In the second case, observe that \eqref{boundcoerc} can be rewritten via \eqref{neharibelong} (with $\beta=\lambda_1(q)$) as
$$ E_{\alpha, \lambda_1(q)}(u_n)=\left(\frac{1}{p}-\frac{1}{q} \right)(\rho_{\theta_0}(\nabla u_n)-\alpha\rho_{\theta_0}(u_n)) \leq M \quad \text{for all }n\in\N. $$
Then, dividing by $\|u_n\|_\theta^p$ gives
$$ 0\leq \frac{E_{\alpha, \lambda_1(q)}(u_n)}{\|u_n\|_\theta^p} = \left(\frac{1}{p}-\frac{1}{q} \right)(\rho_{\theta_0}(\nabla y_n)-\alpha\rho_{\theta_0}(y_n)) \leq \frac{M}{\|u_n\|_\theta^p} \quad \text{for all }n\in\N, $$
whence
$$\lim_{n\to\infty} (\rho_{\theta_0}(\nabla y_n)-\alpha\rho_{\theta_0}(y_n)) = 0.$$
An argument similar to \eqref{Wbound1}--\eqref{Wbound2} entails $y_n\rightharpoonup t\phi_q$ in $W^{1,\theta_0}(\Omega)$, $y_n \rightarrow t\phi_q$ in $L^{\theta_0}(\Omega),$ and $$ \rho_{\theta_0}(\nabla \phi_q) - \alpha\rho_{\theta_0}(\phi_q) \le 0,$$
which contradicts $\alpha<\tilde{s} _{+}.$
Therefore, $E_{\alpha, \lambda_1(q)}\mid_{\mathcal N_{\alpha, \lambda_1(q)}}$ is coercive. 

\underline{Claim 3}: $\displaystyle{m_{\alpha,\lambda_1(q)}:=\min_{ \mathcal{N}_{\alpha, \lambda_1(q)}} E_{\alpha, \lambda_1(q)}>0.}$ \\
Reasoning by contradiction as for Claim 2 of Theorem \ref{nehari}, there exists $\{u_n\}\subseteq \mathcal{N}_{\alpha,\lambda_1(q)}$ and $u\in W^{1,\theta}_0(\Omega)$ such that $u_n\rightharpoonup u$ in $W^{1,\theta}_0(\Omega)$ and
$$ \|\nabla u_n\|_q^q-\lambda_1(q)\|u_n\|_q^q \to 0 \quad \mbox{as} \;\; n\to\infty, $$
whence
$$ \|\nabla u\|_q^q-\lambda_1(q)\|u\|_q^q = 0. $$
Thus, either $u=0$ or $u=t\phi_q$ for some $t\in\R\setminus\{0\}$. In the first case, setting $y_n:=\frac{u_n}{\|u_n\|_\theta}$ and reasoning as in \eqref{yncontrol} yield \eqref{Wbound2}; then the conclusion follows as in Claim 2. On the other hand, if $u=t\phi_q$ then \eqref{neharibelong} gives
$$ \rho_{\theta_0}(\nabla u)-\alpha\rho_{\theta_0}(u) \leq \liminf_{n\to\infty} \left(\rho_{\theta_0}(\nabla u_n)-\alpha\rho_{\theta_0}(u_n)\right) = \lim_{n\to\infty} \left(\lambda_1(q)\|u_n\|_q^q-\|\nabla u_n\|_q^q\right) = 0, $$
so that $\alpha \geq \frac{\rho_{\theta_0}(\nabla u)}{\rho_{\theta_0}(u)}=\tilde{s}_+$, contradicting $\alpha<\tilde{s}_+$.

%
%
Let $\{u_n\}\subseteq \mathcal{N}_{\alpha,\lambda_1(q)}$ be a minimizing sequence for $E_{\alpha,\lambda_1(q)}\mid_{\mathcal{N}_{\alpha,\lambda_1(q)}}$. Coercivity of $E_{\alpha,\lambda_1(q)}\mid_{\mathcal{N}_{\alpha,\lambda_1(q)}}$ ensures that
$$ u_n\rightharpoonup u \quad \text{in }W^{1,\theta}_0(\Omega) \quad \text{and} \quad u_n\to u \quad \text{in }L^\theta(\Omega) $$
for some $u\in W^{1,\theta}_0(\Omega)$. The proof can be concluded as in Theorem \ref{nehari}, provided $u\neq t\phi_q$ for any $t\in\R\setminus\{0\}$, which is necessary to adapt the argument in \cite{GGP}. To this aim, suppose by contradiction $u=t\phi_q$ for some $t\in\R\setminus\{0\}$. Notice that the weak lower semi-continuity of $\|\cdot\|_q$, \eqref{neharibelong} for $\beta=\lambda_1(q)$, and \eqref{rayleigh} yield
\begin{equation*}
\begin{aligned}
t^p\rho_{\theta_0}(\nabla \phi_q) + t^q\|\nabla \phi_q\|_q^q &\leq \liminf_{n\to\infty} (\rho_{\theta_0}(\nabla u_n) + \|\nabla u_n\|_q^q) = \lim_{n\to\infty} (\alpha \rho_{\theta_0}(u_n) + \lambda_1(q)\|u_n\|_q^q) \\
&= \alpha t^p\rho_{\theta_0}(\phi_q) + \lambda_1(q)t^q\|\phi_q\|_q^q \leq \alpha t^p\rho_{\theta_0}(\phi_q) + t^q\|\nabla \phi_q\|_q^q.
\end{aligned}
\end{equation*}
Accordingly, $\rho_{\theta_0}(\nabla \phi_q) \leq \alpha \rho_{\theta_0}(\phi_q)$, contradicting $\alpha<\tilde{s}_+$.
\end{proof}

\begin{prop}
\label{epsilon}
Suppose that  \eqref{LI} holds true. Then  \begin{itemize}
    \item[(i)]  for any $\alpha\in[\lambda_1^a(p),\tilde{s}_{+})$ there exists $\eps_0=\eps_0(\alpha)>0$ such that \GEV{$\alpha,\lambda_1(q)+\eps$} has a positive solution $u\in W^{1,\theta}_0(\Omega)$  for all $\eps\in (0,\eps_0)$;
    \item[(ii)]  for any $\beta\in(\lambda_1(q),\tilde{s}_{-})$ there exists $\eps_0=\eps_0(\beta)>0$ such that \GEV{$\lambda_1^a(p)+\eps,\beta$} admits a positive solution $u\in W^{1,\theta}_0(\Omega)$ for all $\eps\in (0,\eps_0)$.
    \item[(iii)] there exists $\eps_0>0$ such that \GEV{$\lambda_1^a(p)+\eps,\lambda_1(q)+\eps$} admits a positive solution for all $\eps\in(0,\eps_0)$.
    \end{itemize} 
    In all cases, $u$ is a ground-state solution with negative energy.
\begin{proof}
The proofs are inspired by \cite[Section 4]{BT}. 

Let us prove (i). Fix any $\eps>0.$ Since $\alpha<\tilde{s}_{+}$ and $\eps>0$, we have \begin{equation}\label{halpha}\begin{split}
    H_{\alpha}(\phi_q)&>\rho_{\theta_0}(\nabla\phi_q) - \tilde{s}_{+}\rho_{\theta_0}(\phi_q) = 0
\end{split}
    \end{equation}
    and \begin{equation}\label{gbeta}
        G_{\lambda_1(q)+ \eps }(\phi_q)< G_{\lambda_1(q)}(\phi_q)=0.
    \end{equation} Proposition \ref{extremepoint}, besides \eqref{halpha}--\eqref{gbeta}, guarantees  $\mathcal{N}_{\alpha, \lambda_1(q)+\eps} \not = \emptyset$ and
    \begin{equation}
        \label{minnegative}
        m_{\alpha,\lambda_1(q)+\eps}:= \inf_{\mathcal{N}_{\alpha, \lambda_1{(q)+\eps} }}E_{\alpha, \lambda_1(q)+\eps}<0.
    \end{equation}
    Let $\{u_n^{\eps}\} \subseteq \mathcal{N}_{\alpha, \lambda_1{(q)+\eps}}$ be such that $E_{\alpha, \lambda_1(q)+\eps }(u_n^{\eps})\to m_{\alpha, \lambda_1(q)+\eps}$ as $n\to\infty$. According to \eqref{minnegative}, we can assume $$E_{\alpha, \lambda_1(q)+\eps}(u_n^\eps)<0 \quad \text{for all } n\in\N,$$  and consequently $G_{\lambda_1(q)+\eps}(u_n^\eps)<0<H_{\alpha}(u_n^\eps)$ for all $n\in\N$: indeed, recalling that $u_n^{\eps} \in \mathcal{N}_{\alpha, \lambda_1{(q)+\eps}}$ and $q<p$, one has 
\begin{equation*}\begin{split}
     \left(\frac{q-p}{pq}\right)H_{\alpha}(u_n^\eps)&=\left(\frac{p-q}{pq}\right)G_{\lambda_1(q)+\eps}(u_n^\eps)=\frac{1}{p}H_{\alpha}(u_n^\eps)+\frac{1}{q}G_{\lambda_1(q)+\eps}(u_n^\eps)\\&=E_{\alpha,\lambda_1(q)+\eps }(u_n^\eps)<0.
\end{split}
\end{equation*}

\underline{Claim}: there exists $\eps_0 \in (0,1)$ such that, for any $\eps\in(0,\eps_0)$, the sequence $\{u_n^{\eps}\}$ is bounded in $W^{1, \theta}_{0}(\Omega)$.\\
Suppose by contradiction that for every $k\in \N$ there exists $\eps_k \in \left(0, \frac{1}{k}\right)$ such that $u_k:=u_{{n}_k}^{\eps_k}\in \mathcal{N}_{\alpha, \lambda_1(q)+\eps_k }$ satisfies $t_k:=\left\|u_k\right\|_{1,\theta}>k.$ Set $v_k:=\frac{u_k}{\| u_k\|_{1,\theta}}$ for all $k \in \N$. Since $W^{1,\theta}_0(\Omega)\hookrightarrow L^\theta(\Omega)$ compactly, we have \begin{equation}\label{convergence}
v_k\rightharpoonup v^{*} \quad \text{in } W^{1, \theta}_{0}(\Omega) \quad \text{and} \quad v_k \rightarrow v^{*} \quad \text{in } L^{\theta}(\Omega).
\end{equation}

Let us show that $v^*\neq 0$.
Since $u_k\in \mathcal{N}_{\alpha, \lambda_1(q)+\eps_k }$ for all $k\in\N$, we have $$|H_{\alpha}( u_k)|=|G_{ \lambda_1(q)+\eps_k}( u_k)| \quad \text{for all }k \in \N.$$ Dividing by $t_k^p,$ we get  $$|H_{\alpha}( v_{k})|=t_k^{q-p}|G_{ \lambda_1(q)+\eps_k}( v_{k})| \quad \text{for all }k \in \N.$$ 
According to \eqref{convergence} and $\|\nabla v_k\|_q \leq \|v_k\|_{1,\theta} = 1$, besides $G_{\lambda_1(q)+\eps_k}(v_k)<0$ for every $k \in \N$, and letting $k \to \infty$ entail \begin{equation}\label{liminfHG}
	 \lim_{k\to \infty}H_{\alpha}(v_k)=0 \quad \text{and} \quad \limsup_{k\to \infty}G_{\lambda_1(q)+\eps_k}(v_k)\le0. 
	\end{equation} Thus, we infer $H_{\alpha}(v^{*})\le 0$ and $G_{\lambda_1(q)}(v^*)\le 0:$ indeed, by \eqref{convergence} and the weak lower semi-continuity of both $H_\alpha$ and $G_\beta$, \begin{equation}\label{H}
	H_{\alpha}(v^{*})\le \lim_{k\to \infty}H_{\alpha}(v_k)=0
\end{equation}
and
\begin{equation}\label{G}
G_{\lambda_1(q)}(v^*)\le \liminf_{k\to \infty}G_{\lambda_1(q)+\eps_k}(v_k)\le \limsup_{k\to \infty}G_{\lambda_1(q)+\eps_k}(v_k)\le 0.
\end{equation}
Since $\|v_k\|_{1,\theta}=1$ for all $k\in\N$, besides exploiting \eqref{convergence} and \eqref{liminfHG}--\eqref{H}, we deduce
\begin{equation}\begin{split}\label{vweight}
\alpha \rho_{\theta_0}(v^*) + \lambda_1(q)\|v^*\|_q^q &= \lim_{k\to\infty} (\alpha \rho_{\theta_0}(v_k) + (\lambda_1(q)+\eps_k)\|v_k\|_q^q) \\
&=1- \lim_{k\to \infty}(H_{\alpha}(v_k)+G_{\lambda_1(q)+\eps_k}(v_k))\\
&\ge 1- \lim_{k\to\infty} H_{\alpha}( v_{k})-\limsup_{k\to\infty}  G_{\lambda_1(q)+\eps_k}( v_{k})\ge 1.
\end{split}
\end{equation}
Thus $v^{*}\neq 0.$

By \eqref{rayleigh} we have $G_{\lambda_1(q)}(v^*)\geq 0$, so \eqref{G} yields $v^*= t \phi_q$ for some $t\in \R\setminus\{0\}.$ Accordingly, $\alpha<\tilde{s}_{+}$ forces $H_{\alpha}(t \phi_q)>0,$ which contradicts \eqref{H}. The claim is proved.

Fix any $\eps\in(0,\eps_0)$ and let $u_0^\eps \in W^{1, \theta}_0(\Omega)$ be such that $u_n^\eps\rightharpoonup u_0^\eps$ in $W^{1, \theta}_0(\Omega)$. Observe that $u_0^\eps\neq 0$, since the weak lower semi-continuity of $E_{\alpha,\lambda_1(q)+\eps}$ guarantees
$$E(u_0^\eps)\leq \lim_{n\to\infty} E_{\alpha, \lambda_1(q)+\eps}(u_n^\eps) = m_{\alpha, \lambda_1(q)+\eps}<0=E_{\alpha, \lambda_1(q)+\eps}(0).$$
Arguing as in \cite[Lemma 4]{BT}, we have
\begin{equation*}
     G_{\lambda_1(q)+\eps}(u_0^\eps)<0< H_{\alpha}( u_0^\eps) \quad \text{for all } \eps \in (0, \eps_0),
\end{equation*}
taking a smaller $\eps_0$ if necessary. Then, following \cite[Proposition 7 and Lemma 5]{BT}, we deduce $u_0^\eps \in \mathcal{N}_{\alpha, \lambda_1{(q)+\eps} }$ and that the minimizer $m_{\alpha,\lambda_1(q)+\eps} $ of $E_{\alpha,\lambda_1(q)+\eps}\mid_{\mathcal{N}_{\alpha,\lambda_1(q)+\eps}}$ is attained in $u_0^\eps$ for any $\eps \in (0, \eps_0).$ Reasoning as in Theorem \ref{nehari}, we conclude that $u_0^\eps$ is a positive solution to \GEV{$\alpha,\lambda_1(q)+\eps$} for any $\eps \in (0, \eps_0)$.

The proof of (ii) is analogous to (i), except for the last part of the claim; thus, we discuss only how to reach the contradiction establishing the claim. \\
Repeating verbatim the proof (i) until \eqref{vweight}, we get $v^*\neq 0$, $H_{\lambda_1^a(p)}(v^*)\leq 0$, and $G_\beta(v^*)\leq 0$. Since $H_{\lambda_1^a(p)}(v^*)\geq 0$ by \eqref{rayleigh}, we deduce $v^*=t\phi_1^a(p)$ for some $t\in\R\setminus\{0\}$. If $\tilde{s}_-<+\infty$, then $G_\beta(t\phi_1^a(p))\leq 0$, contradicting $\beta<\tilde{s}_-$; otherwise, $v^*=t\phi_1^a(p)$ forces $\phi_1^a(p)\in W^{1,\theta}_0(\Omega)$, against the assumption that $\int_\Omega |\nabla \phi_1^a(p)|^q \dx$ is not finite.

To prove (iii), observe that there exists $\eps_0>0$ such that
$$ G_{\lambda_1(q)+\eps_0}(\phi_q) < 0 < H_{\lambda_1^a(p)+\eps_0}(\phi_q). $$
Thus, arguing as in (i) produces a positive solution to \GEV{$\lambda_1^a(p)+\eps,\lambda_1(q)+\eps$} for any $\eps\in(0,\eps_0)$, restricting $\eps_0$ if necessary.
\end{proof}
\end{prop}

\begin{prop}
\label{lambdaprop}
The function $\lambda^{*}$ defined in \eqref{definitionlambda} satisfies the following properties:
\begin{enumerate}[label={\rm (\roman*)}]
\item  $\lambda^{*}(s)<+\infty$ for any $s \in \R;$
\item  $\lambda^{*}(s)+s\ge \lambda_{1}^{a}(p)$ and  $\lambda^{*}(s)\ge \lambda_{1}(q)$ for all $s \in \R;$
\item   $\lambda^{*}(s)$ is non-increasing and  $\lambda^{*}(s)+s$ is non-decreasing on $\R;$
\item $\lambda^{*}(s)$ is continuous on $\R;$
\item   $\lambda^{*}(s)= \lambda_{1}(q)$ for any $s\ge s^*_{+};$
\item $\lambda^*(s^*)+s^*>\lambda_1^a(p)$ and $\lambda^{*}(s^{*})>\lambda_{1}(q)$ if and only if \eqref{LI} holds true.
\end{enumerate}
\begin{proof}
\textbf{(i)} Fix $\lambda,s\in \R$  and suppose that there exists  $u \in  W^{1, \theta}_{0}(\Omega)$  positive solution to \GEV{$\lambda+s,\lambda$}. Without loss of generality we can assume $\lambda$ so large that $\lambda+s>0$. Fix any $K\Subset\Omega$ and set $m_K:=\min_K a > 0$. Moreover, take any non-negative, non-trivial function $\phi\in C^\infty_c(K)$. Owing to Proposition \ref{strongmax} we have $\essinf_K u>0$, so
$$\xi:=\frac{\phi^p}{m_Ku^{p-1}+u^{q-1}}\in W^{1, \theta}_{0}(\Omega)$$
is admissible as test function in \GEV{$\lambda+s,\lambda$}. Therefore, recalling that $\lambda+s>0$,
\begin{equation}\label{test}
\begin{split}
&\int_{\Omega} a|\nabla u|^{p-2} \nabla u \cdot \nabla \xi\dx +\int_{\Omega} |\nabla u|^{q-2} \nabla u \cdot \nabla \xi\dx \\
&=(\lambda+s) \int_{\Omega}au^{p-1}\xi\dx+\lambda \int_{\Omega}u^{q-1}\xi\dx\\
&\ge (\lambda+s) m_{K} \int_{K} u^{p-1}\xi\dx+\lambda \int_{K}u^{q-1}\xi\dx\\
&= \lambda  \int_{K}\phi^p\dx+ s  \int_{K} \frac{ m_{K} u^{p-1}\phi^p}{m_K u^{p-1}+u^{q-1}}\dx  \ge  (\lambda-|s|)  \int_{K}\phi^p\dx.
\end{split}
\end{equation}
Moreover, a slight adaptation of \cite[Proposition 8]{BT} produces $\rho>0$ such that
\begin{equation}\label{picone}
\begin{split}
&\int_{\Omega} a|\nabla u|^{p-2} \nabla u \cdot \nabla \xi\dx +\int_{\Omega} |\nabla u|^{q-2} \nabla u \cdot \nabla \xi\dx\\
&\le \frac{1}{\rho} \left( \int_{K} a |\nabla \phi|^p\dx +\int_{K} |\nabla (\phi^{\frac{p}{q}})|^q\dx \right).
\end{split}
\end{equation}
Putting \eqref{test}--\eqref{picone} together gives
\begin{equation*}
\begin{split}
  \lambda \le \frac{1}{\rho \int_{K}\phi^p\dx} \left(\int_{K}a|\nabla \phi|^p\dx +\int_{K} |\nabla (\phi^{\frac{p}{q}})|^q\dx \right) + |s|.
\end{split}
\end{equation*}
Since $\phi$, $K$, and $s$ are independent of $\lambda$, we deduce that $\lambda^*(s)$ is finite.

\textbf{(ii)} Let $s\in \R.$ Three circumstances can occur. 
\begin{itemize}
\item If $s<s^*,$ then $\lambda_1(q)< \lambda_1^a(p)-s.$ Take any $\lambda\in (\lambda_1(q),\lambda_1^{a}(p)-s)$. Then Theorem \ref{normalize}  ensures that \GEV{$\lambda+s,\lambda$} admits a positive solution. Hence  $\lambda^*(s)>\lambda_1(q)$ and, by arbitrariness of $\lambda$, $\lambda^*(s)+s\ge\lambda^a_1(p)$.
\item Let $s>s^*.$  For any $\lambda\in (\lambda_1^{a}(p)-s,\lambda_1(q)),$ Theorem \ref{nehari}  provides the existence of a positive solution to \GEV{$\lambda+s,\lambda$}. Therefore  $\lambda^*(s)+s>\lambda^a_1(p)$ and, since $\lambda$ was arbitrary, $\lambda^*(s)\ge\lambda_1(q)$.
\item Let $s=s^*$. If \eqref{LI} holds true then, according to Proposition \ref{epsilon}(iii), \GEV{$\lambda_1^a(p)+\eps,\lambda_1(q)+\eps$} admits a positive solution provided $\eps$ is sufficiently small. Then $\lambda^*(s^*)+s^*>\lambda_1^a(p)$ and $\lambda^*(s^*)>\lambda_1(q)$. If \eqref{LI} does not hold, then Theorem \ref{notex} ensures the existence of a positive solution to \GEV{$\lambda_1^a(p),\lambda_1(q)$}, so $\lambda^*(s^*)+s^*\geq \lambda_1^a(p)$ and $\lambda^*(s^*)\geq \lambda_1(q)$.
\end{itemize}

\textbf{(iii)} Fix any $s,s'\in\R$ fulfilling $s<s'$. We prove that $\lambda^*(s')\le \lambda^{*}(s)$. From (ii) we have $\lambda^*(s'),\lambda^*(s)\geq \lambda_1(q)$. Thus, without loss of generality, we can suppose that $\lambda^*(s')>\lambda_1(q).$ Choose any $\eps>0$ satisfying $\lambda^*(s')-\eps>\lambda_1(q)$. Then there exists $\nu \in (\lambda^*(s')-\eps,\lambda^*(s'))$ such that \GEV{$\nu+s',\nu$} admits a positive solution $w \in W^{1, \theta}_0(\Omega)$.

Let  $\{\Omega_n\}$ be a sequence of $C^2$ bounded domains such that $\Omega_n\Subset\Omega$ and $\Omega_n \nearrow \Omega.$ For all $n \in  \N$, denote with $(\lambda_1(q, n),\phi_{q,n})$ the first eigenpair of $(-\Delta_q,W^{1,q}_0(\Omega_n))$, being $\|\phi_{q,n}\|_{L^\infty(\overline{\Omega}_n)}=1$ for all $n\in\N$. Each $\phi_{q,n}$  can be extended to $\Omega$ by setting $\phi_{q,n}\equiv 0$ outside $\Omega_n.$ Observe that
\begin{equation}
\label{eigencontinuity}
\lim_{n \to \infty}\lambda_1(q, n)=\lambda_1(q).
\end{equation}
Indeed, using \eqref{rayleigh} we get $\lambda_1(q,n)\ge \lambda_1(q)$ for all $n\in\N$. On the other hand, taking any $\{\xi_n\}\subseteq C^\infty_c(\Omega)\setminus\{0\}$ such that $\supp \xi_n \subseteq \Omega_n$ for all $n\in\N$ and $\xi_n \to \phi_q$ in $W^{1,q}_0(\Omega)$, we have
\begin{equation}
\label{undertheinf}
\lambda_1(q,n) = \frac{\|\nabla \phi_{q,n}\|_q^q}{\|\phi_{q,n}\|_q^q} \leq \frac{\|\nabla \xi_n\|_q^q}{\|\xi_n\|_q^q} \to \frac{\|\nabla \phi_q\|_q^q}{\|\phi_q\|_q^q} = \lambda_1(q) \quad \text{as }n\to\infty,
\end{equation}
ensuring $\lambda_1(q,n) \leq \lambda_1(q)$.

Thus we can choose $n\in\N$ such that $\lambda_1(q,n)<\nu$ and we set $m:=\essinf_{\Omega_n} w$, which is positive according to Proposition \ref{strongmax}. Thus, there exists $\overline{t}>0$ such that $w\geq m \geq \overline{t}\phi_{q,n}$ in $\Omega_n$. Let $T_w:\Omega\times \R\to \R$ be defined as $T_w(x,t):=\min\{t_+,w(x)\}$ and consider the following problem:
\begin{equation}
		\label{nuprob}
		\tag{${\rm P}_\nu$}
		\left\{
		\begin{alignedat}{2}
			-\Delta_{p}^a u -\Delta_{q}u &= (\nu+s)a (T_w(x,u))^{p-1} + \nu (T_w(x,u))^{q-1} &&\quad \mbox{in}\;\; \Omega, \\
			u &=0 &&\quad \mbox{on}\;\; \partial \Omega.
		\end{alignedat}
		\right.
	\end{equation}	
 The energy functional associated to \eqref{nuprob} is 
$$J(u):= \frac{1}{p}\rho_{\theta_0}(\nabla u) + \frac{1}{q}\|\nabla u\|_q^q - (\nu+s) \int_\Omega a F_p(x,u) \dx - \nu \int_\Omega F_q(x,u) \dx$$
for all $u \in W^{1,\theta}_0(\Omega),$ being $F_r(x,t):=\int_0^t (T_w(x,\tau))^{r-1} \dtau$ for $r=p,q$. Notice that $J$ is well defined and weakly sequentially lower semi-continuous. Notice that
\begin{equation*}
F_r(x,t) = \frac{1}{r} (T_w(x,t))^r + w(x)^{r-1}(t-w(x))_+.
\end{equation*}
By Young's inequality with $\sigma>0$ we get $F_r(x,t)\leq C_\sigma (w(x))^r + \sigma |t|^r $, so using also \eqref{rayleigh} yields
\begin{equation}
\label{trunccoerc}
\begin{aligned}
J(u) &\geq \frac{1}{p}\rho_{\theta_0}(\nabla u) + \frac{1}{q}\|\nabla u\|_q^q \\
&\quad - |\nu+s|C_\sigma \rho_{\theta_0}(w) - |\nu+s|\sigma \rho_{\theta_0}(u) - \nu C_\sigma \|w\|_q^q - \nu\sigma \|u\|_q^q \\
&\geq \left[\frac{1}{p}-\sigma\frac{|\nu+s|}{\lambda_1^a(p)}\right] \rho_{\theta_0}(\nabla u) + \left[\frac{1}{q}-\sigma\frac{\nu}{\lambda_1(q)}\right]\|\nabla u\|_q^q - C_\sigma \rho_\theta(w),
\end{aligned}
\end{equation}
enlarging $C_\sigma$ if necessary. Taking $\sigma>0$ small enough guarantees $J(u)\geq \sigma\rho_\theta(\nabla u)-C_\sigma \rho_\theta(w)$, so that $J$ is coercive. It follows that there exists $u\in W^{1,\theta}_0(\Omega)$ global minimizer of $J$. Notice that $F_r(x,t\phi_{q,n}(x))\geq \frac{t^r}{r}(\phi_{q,n}(x))^r\geq 0$ for all $(x,s)\in\Omega\times \R$ and all $t\in(0,\overline{t})$, whence
\begin{equation*}
\begin{aligned}
J(u)&\leq J(t\phi_{q,n}) \leq \frac{t^p}{p}\rho_{\theta_0}(\nabla \phi_{q,n}) + \frac{t^q}{q}\|\nabla \phi_{q,n}\|_q^q + |\nu+s|\frac{t^p}{p}\rho_{\theta_0}(\phi_{q,n}) - \nu\frac{t^q}{q} \|\phi_{q,n}\|_q^q \\
&\leq \frac{t^p}{p}\left(\rho_{\theta_0}(\nabla \phi_{q,n}) + |\nu+s|\rho_{\theta_0}(\phi_{q,n})\right) + \frac{t^q}{q} \left(\lambda_1(q,n)-\nu\right) \|\phi_{q,n}\|_q^q < 0
\end{aligned}
\end{equation*}
for all $t\in(0,\overline{t})$ sufficiently small. We deduce that $u\neq 0$. Since $\nu+s<\nu+s'$, the truncation levels $0$ and $w$ are sub- and super-solution to \eqref{nuprob}, respectively. Thus, testing with $u_-$ and $(u-w)_+$, we obtain $0\leq u\leq w$ in $\Omega$. By definition of $T_w$, $u$ is a non-negative, non-trivial solution to \GEV{$\nu+s,\nu$}. Owing to Proposition \ref{strongmax}, $u>0$ in $\Omega$. Accordingly, $u$ is a positive solution to \GEV{$\nu+s,\nu$}. Hence, from the definition of $\lambda^*,$ we get  $\nu\le \lambda^*(s).$ Recalling that  $\lambda^*(s')-\eps<\nu,$ we deduce $\lambda^*(s')<\lambda^*(s)+\eps$. Arbitrariness of $\eps$ permits to conclude. An argument analogous to \cite[Proposition 3(vi)]{BT} ensures the monotonicity of $\lambda^*(s)+s$.

\textbf{(iv)}  Continuity follows by the monotonicity proved in (iii), as in \cite[Proposition 3(v)]{BT}.

\textbf{(v)} Retaining the notation of (iii), we define $$s^{*}_{+,n}=\frac{ \rho_{\theta_{0}}( \nabla \phi_{q,n})}{\rho_{\theta_0} (\phi_{q,n})}-\lambda_1(q,n).$$

\underline{Claim}: $\displaystyle{\lim_{n\to \infty} s^{*}_{+,n}= s^{*}_{+}}.$ \\
According to \eqref{eigencontinuity}, it suffices to prove
\begin{equation}
\label{rayleighconv}
\lim_{n\to \infty} \frac{\rho_{\theta_0}(\nabla \phi_{q,n})}{\rho_{\theta_0}(\phi_{q,n})} = \frac{\rho_{\theta_0}(\nabla \phi_q)}{\rho_{\theta_0}(\phi_q)}.
\end{equation}

Fix any $n\in\N$. Since $\Omega$ and $\Omega_n$ are of class $C^2$, there exists a $C^2$-diffeomorphism $\Upsilon_n:\overline{\Omega}\to\overline{\Omega}_n$ such that $\|\Upsilon_n-I\|_{C^2(\overline{\Omega})} \to 0$ and $\|\Upsilon_n^{-1}-I_n\|_{C^2(\overline{\Omega}_n)} \to 0$ as $n\to\infty$, being $I$ and $I_n$ the identity functions on $\overline{\Omega}$ and $\overline{\Omega}_n$, respectively. Let us define $\tilde{\phi}_{q,n}:=\phi_{q,n}\circ \Upsilon_n$, which belongs to $C^{1,\tau}(\overline{\Omega})$. Moreover, since $\|\phi_{q,n}\|_{L^\infty(\overline{\Omega}_n)} = 1$ for all $n\in\N$, then  $\|\phi_{q,n}\|_{C^{1,\tau}(\overline{\Omega}_n)} \leq C$ for all $n\in\N$ by nonlinear regularity theory \cite{LI}, being $C>0$ opportune. As a consequence, $\{\tilde{\phi}_{q,n}\}$ is bounded in $C^{1,\tau}(\overline{\Omega})$. In particular, $\tilde{\phi}_{q,n}\to \hat{\phi}$ in $C^1(\overline{\Omega})$ for some $\hat{\phi}\in C^1(\overline{\Omega})\setminus\{0\}$, due to Ascoli-Arzelà's theorem and $\|\tilde{\phi}_{q,n}\|_{L^\infty(\overline{\Omega})}=1$ for all $n\in\N$ by construction. Thus, by uniform convergence,
\begin{equation*}
\begin{aligned}
\int_\Omega |\nabla \phi_{q,n}|^q \dx &= \int_\Omega |\nabla (\tilde{\phi}_{q,n}\circ \Upsilon_n^{-1})|^q \dx \\
&= \int_\Omega |[(\nabla \tilde{\phi}_{q,n})\circ \Upsilon_n^{-1}] \Jac(\Upsilon_n^{-1})|^q \dx  \to \int_\Omega |\nabla \hat{\phi}|^q \dx.
\end{aligned}
\end{equation*}
The same argument yields
\begin{equation*}
\int_\Omega \phi_{q,n}^q \dx \to \int_\Omega \hat{\phi}^q \dx > 0.
\end{equation*}
Hence, exploiting \eqref{undertheinf}, one has $\frac{\|\nabla \hat{\phi}\|_q^q}{\|\hat{\phi}\|_q^q} = \lambda_1(q)$, forcing $\hat{\phi}=t\phi_q$ for some $t>0$. Analogously,
\begin{equation*}
\begin{aligned}
\frac{\int_\Omega a|\nabla \phi_{q,n}|^p \dx}{\int_\Omega a|\phi_{q,n}|^p \dx} \to \frac{\int_\Omega a|\nabla \phi_q|^p \dx}{\int_\Omega a|\phi_q|^p \dx},
\end{aligned}
\end{equation*}
which entails \eqref{rayleighconv}.

Fixed any $s> s^*_{+},$ we prove that $\lambda^*(s)=\lambda_1(q).$ Taking into account (ii), we suppose by contradiction that there exists $s> s^*_{+}$ such that $\lambda^*(s)>\lambda_1(q).$ Take any $\eps \in (0,1)$ fulfilling $\lambda^*(s)-\eps >\lambda_1(q).$  By definition of $\lambda^{*}$, taking a smaller $\eps$ if necessary, \GEV{$\lambda^*(s)-\eps+s, \lambda^*(s)-\eps$} admits a positive solution $u\in W^{1, \theta}_{0}(\Omega).$ According to \eqref{eigencontinuity} and recalling that $s^{*}_{+,n}\to  s^{*}_{+}$, we can choose $n\in\N$ such that
\begin{equation}
\label{room}
\lambda^*(s)-\eps>\lambda_1(q,n)\quad \text{and} \quad s>s^{*}_{+,n}.
\end{equation}
Owing to Proposition \ref{strongmax}, one has $\essinf_{\Omega_n} u>0$, ensuring $\frac{\phi_{q,n}}{u}\in L^\infty(\Omega_n)$. Hence, testing \GEV{$\lambda^*(s)-\eps+s, \lambda^*(s)-\eps$} with $\frac{\phi_{q, n}^p}{u^{p-1}}\in W^{1,\theta}_0(\Omega)$, besides exploiting \eqref{room}, we deduce
\begin{equation}\label{deltaine}
\begin{split}
&\int_{\Omega} a|\nabla u|^{p-2} \nabla u \cdot \nabla \left(\frac{\phi_{q, n}^p}{u^{p-1}}\right)\dx +\int_{\Omega} |\nabla u|^{q-2} \nabla u \cdot \nabla \left(\frac{\phi_{q, n}^p}{u^{p-1}}\right)\dx\\
&=\left(\lambda^*(s)-\eps+s\right) \int_{\Omega_{n}}a\phi_{q, n}^p\dx +\left(\lambda^*(s)-\eps\right)\int_{\Omega_n} u^{q-1}  \frac{\phi_{q, n}^p}{u^{p-1}}\dx\\
&>\left(\lambda_1(q, n)+s\right)\int_{\Omega_{n}}a\phi_{q, n}^p\dx +\lambda_1(q,n)\int_{\Omega_{n}} u^{q-p}  \phi_{q,n}^p\dx.\\
\end{split}
\end{equation} 
On the other hand, Picone's inequalities \cite[Theorem 1.1]{A} and \cite[Lemma 1]{I} with $\gamma=q$ ensure 
\begin{equation}\label{piconeine}
\begin{split}
&\int_{\Omega} a|\nabla u|^{p-2} \nabla u \cdot \nabla \left(\frac{ \phi_{q, n}^p}{u^{p-1}}\right)\dx +\int_{\Omega} |\nabla u|^{q-2} \nabla u \cdot \nabla \left(\frac{ \phi_{q, n}^p}{u^{p-1}}\right)\dx\\
&\le \int_{\Omega_n} a|\nabla \phi_{q, n}|^{p} \dx +\lambda_1(q, n) \int_{\Omega_{n}}  \phi_{q, n}^p u^{q-p}\dx.\\
\end{split}
\end{equation}  
Combining \eqref{deltaine}--\eqref{piconeine} we get
\begin{equation*}
\left(\lambda_{1}(q, n)+s\right)\int_{\Omega_{n}} a\phi_{q, n}^{p} \dx<\int_{\Omega_{n}}a|\nabla \phi_{q, n}|^{p} \dx,
\end{equation*}
contradicting \eqref{room}. Hence $\lambda^*(s)=\lambda_1(q)$ for all $s>s^*_+$.

The continuity of $\lambda^{*}$ proved in (iv) forces $\lambda^{*}(s^{*}_{+})=\lambda_1(q).$

\textbf{(vi)} If \eqref{LI} holds true, the conclusion follows as in the proof of (ii). If \eqref{LI} does not hold, we suppose by contradiction that \eqref{prob} admits a positive solution for some $(\alpha,\beta)$ fulfilling
\begin{equation}
\label{contradiction}
\alpha>\lambda_1^a(p) \quad \text{and} \quad \beta>\lambda_1(q).
\end{equation}
Set $\tilde{s}_{+,n}:=\frac{\rho_{\theta_0}(\nabla \phi_{q,n})}{\rho_{\theta_0}(\phi_{q,n})}$ for all $n\in\N$. Since $\phi_q=k\phi_p^a$ for some $k\neq 0$, \eqref{rayleighconv} gives
$$ \lim_{n\to\infty} \tilde{s}_{+,n} = \frac{\rho_{\theta_0}(\nabla \phi_q)}{\rho_{\theta_0}(\phi_q)} = \frac{\rho_{\theta_0}(\nabla \phi_p^a)}{\rho_{\theta_0}(\phi_p^a)} = \lambda_1^a(p). $$
Then, reasoning as in \eqref{deltaine}--\eqref{piconeine}, we get
\begin{equation*}
\begin{aligned}
\alpha \int_\Omega a\phi_{q,n}^p \dx + \beta \int_\Omega \phi_{q,n}^p u^{q-p} \dx \leq \tilde{s}_{+,n} \int_\Omega a \phi_{q,n}^p \dx + \lambda_1(q,n) \int_\Omega \phi_{q,n}^p u^{q-p} \dx,
\end{aligned}
\end{equation*}
which forces either $\alpha\leq \tilde{s}_{+,n}$ or $\beta \leq \lambda_1(q,n)$. This contradicts \eqref{contradiction}: indeed, recalling that $\tilde{s}_{+,n}\to \lambda_1^a(p)$ and $\lambda_1(q,n)\to\lambda_1(q)$, we can choose $n\in\N$ such that $\alpha>\tilde{s}_{+,n}$ and $\beta>\lambda_1(q,n)$.
\end{proof}

\end{prop}
\begin{thm}\label{connessione}
Assume \eqref{LI} and set $s:=\alpha-\beta$. Then, for all $\alpha\in [\lambda_1^a(p), \tilde{s}_+)$ and $\beta\in(\lambda_1(q),\lambda^*(s))$, \eqref{prob} admits a positive solution.
\begin{proof}
Since $ \beta< \lambda^*(s)$, there exists $\nu \in (\beta,  \lambda^*(s))$ such that  \GEV{$\nu +s, \nu$} admits a positive solution $w\in W^{1, \theta}_{0}(\Omega),$ which is a super-solution to \eqref{prob}. the conclusion follows as in Proposition \ref{lambdaprop}(iii), recalling that $\beta>\lambda_1(q)$.
\end{proof}
\end{thm}

To complete the picture, we discuss the existence of solutions on the curve $\mathcal{C}$ (see Figure \ref{specfig}). We premise a lemma, which is patterned after \cite[Lemma 8]{BT}.
\begin{lemma}
\label{Picone2}
Let $u \in W^{1, \theta}_{0}(\Omega)$ be positive a solution to \eqref{prob}. Then,  for any $ \phi \in W^{1, \theta}_{0}(\Omega) \cap L^\infty(\Omega)$ with $\phi \ge 0,$ one has 
 \begin{equation*}
\int_{\Omega}( \alpha a u^{p-q}+ \beta)\phi^q\dx\le \int_{\Omega} a|\nabla \phi|^q |\nabla u|^{p-q}\dx + \int_{\Omega} |\nabla \phi|^q\dx.
\end{equation*}
\begin{proof}
Take any $ \delta\in (0,1)$ and $\phi\in W^{1, \theta}_{0}(\Omega)$ with $\phi \ge 0.$ Set $\xi_\delta:=\frac{\phi^q}{(u+\delta)^{q-1}}\in W^{1, \theta}_{0}(\Omega)$, since $\frac{\phi}{u+\delta}\in L^\infty(\Omega)$. According to \cite[Proposition 2.9]{BF} we have
\begin{equation}
\label{hiddenconv}
\int_{\Omega} a |\nabla (u+\delta)|^{p-2} \nabla(u+\delta) \cdot \nabla \xi_\delta \dx \leq \int_\Omega a |\nabla\phi|^{q}|\nabla (u+\delta)|^{p-q} \dx.
\end{equation}
Testing \eqref{prob} with $\xi_\delta$, besides using \eqref{hiddenconv} and \cite{A}, yields
\begin{equation*}
\begin{split}
&\alpha\int_{\Omega}au^{p-1} \xi_\delta \dx +\beta \int_{\Omega} u^{q-1}  \xi_\delta \dx \\
&= \int_{\Omega} a|\nabla u|^{p-2} \nabla u \cdot \nabla \xi_\delta \dx +\int_{\Omega} |\nabla u|^{q-2} \nabla u \cdot \nabla \xi_\delta \dx\\
&= \int_{\Omega} a |\nabla (u+\delta)|^{p-2} \nabla(u+\delta) \cdot \nabla \xi_\delta \dx +\int_{\Omega} |\nabla (u+\delta)|^{q-2} \nabla (u+\delta) \cdot \nabla \xi_\delta \dx\\
&\le \int_\Omega a |\nabla\phi|^{q}|\nabla (u+\delta)|^{p-q} \dx + \int_{\Omega} | \nabla \phi|^{q}\dx\\
&= \int_\Omega a |\nabla\phi|^{q}|\nabla u|^{p-q} \dx + \int_{\Omega} | \nabla \phi|^{q}\dx.
\end{split}
\end{equation*}  
Hence Fatou's Lemma permits to conclude.
\end{proof}
\end{lemma}

\begin{thm}
\label{onthecurve}
Suppose \eqref{LI} to be satisfied and let $s\in\R$. If $\lambda^*(s)+s>\lambda_1^a(p)$ and $\lambda^*(s)>\lambda_1(q)$, then \GEV{$\lambda^*(s)+s, \lambda^*(s)$} has at least a positive solution.
\begin{proof}
Set $(\alpha,\beta):=(\lambda^*(s)+s,\lambda^*(s))$. By definition of $\lambda^*$ there exists $\{(\alpha_n, \beta_n)\}\subseteq\R^2$ such that  $\beta_n\to \beta$, $\lambda_1(q)<\beta_n<\beta,$ $\alpha_n= \beta_n +s$, and \GEV{$\alpha_n, \beta_n$} admits a positive solution $u_n \in W^{1,\theta}_{0}(\Omega)$ for all $n \in \N.$  By construction, $\alpha_n \to \alpha> \lambda^a_1(p)$. An adaption of \cite[Lemma 7]{BT} ensures the boundedness of $\{u_n\}$ in $W^{1, \theta}_{0}(\Omega).$ Thus, without loss of generality, we can assume that $\|u_n\|_{1,\theta}\le 1 $ for all $n\in\N$. The compactness of the embedding $W^{1,\theta}_0(\Omega)\hookrightarrow L^\theta(\Omega)$ furnishes $u\in W^{1,\theta}_0(\Omega)$ such that
\begin{equation}
\label{compact}
u_n\rightharpoonup u \quad \text{in }W^{1,\theta}_0(\Omega) \quad \text{and} \quad u_n\to u \quad \text{in }L^\theta(\Omega).
\end{equation}

Consider the functions $\Phi,\Psi:W^{1,\theta}_0(\Omega)\to\R$ defined as 
$$\Phi(u)=\rho_\theta(\nabla u) \quad \text{and} \quad \Psi_{\alpha,\beta}(u)=\alpha\rho_{\theta_0}(u)+\beta\|u\|_q^q \quad \text{for all }u\in W^{1,\theta}_0(\Omega).$$
Thus,
$$ \langle \Phi'(u),\phi \rangle = \int_\Omega a|\nabla u|^{p-2}\nabla u \cdot \nabla \phi \dx + \int_\Omega |\nabla u|^{q-2}\nabla u  \cdot \nabla \phi \dx $$
and
$$ \langle \Psi'_{\alpha,\beta}(u),\phi \rangle = \alpha \int_\Omega a u^{p-1}\phi \dx + \beta \int_\Omega u^{q-1}\phi \dx $$
for all $\phi\in W^{1,\theta}_0(\Omega)$. Now fix any $n\in\N$. Since $u_n\in W^{1,\theta}_{0}(\Omega)$ solves \GEV{$\alpha_n, \beta_n$}, then $\Phi'(u_n)-\Psi'_{\alpha_n,\beta_n}(u_n)=E'_{\alpha_n,\beta_n}(u_n)=0$ in $W^{1,\theta}_0(\Omega)^*$. Therefore,
\begin{equation*}
\begin{aligned}
\langle \Phi'(u_n),u_n-u\rangle &= \langle \Phi'(u_n)-\Psi'_{\alpha_n,\beta_n}(u_n),u_n-u\rangle + \langle \Psi_{\alpha_n,\beta_n}'(u_n),u_n-u \rangle \\
&= \langle E'_{\alpha_n,\beta_n}(u_n),u_n-u\rangle + \langle \Psi_{\alpha_n,\beta_n}'(u_n),u_n-u \rangle \\
&= \langle \Psi_{\alpha_n,\beta_n}'(u_n),u_n-u \rangle.
\end{aligned}
\end{equation*}
Owing to \eqref{compact} we deduce $\Psi_{\alpha_n,\beta_n}'(u_n)\to \Psi'_{\alpha,\beta}(u)$ in $W^{1,\theta}_0(\Omega)^*$. Thus,
$$ \lim_{n\to\infty} \langle \Phi'(u_n),u_n-u\rangle = \lim_{n\to\infty} \langle \Psi_{\alpha_n,\beta_n}'(u_n),u_n-u \rangle = 0. $$
By the ${\rm (S_+)}$ property of $\Phi'$ (see \cite[Theorem 3.3]{CGHW}) we have
\begin{equation}
\label{strongconv}
u_n\to u \quad \text{in }W^{1,\theta}_0(\Omega).
\end{equation}
Hence,
$$ E'_{\alpha,\beta}(u) = \lim_{n\to\infty} E'_{\alpha_n,\beta_n}(u_n) = 0. $$
Accordingly, $u$ is a critical point of $E_{\alpha,\beta}$.
 
It remains to show that $u\neq 0$. By contradiction, suppose $u=0$. Since $u_n$ is a positive solution to \GEV{$\alpha_n, \beta_n$}, using Lemma \ref{Picone2} with $\phi = \phi_q$ entails \begin{equation*}
\int_{\Omega}( \alpha_n a u_n^{p-q}+ \beta_n)\phi_q^q\dx\le \int_{\Omega} a |\nabla \phi_q|^q |\nabla u_n|^{p-q}\dx + \int_{\Omega} |\nabla \phi_q|^q\dx.
\end{equation*}
Letting $n\to\infty$ via Lebesgue's theorem and \eqref{strongconv} produces
 \begin{equation*}
\beta\int_{\Omega} \phi_q^q\dx\le \int_{\Omega} |\nabla \phi_q|^q\dx,
\end{equation*}
which contradicts $ \beta >\lambda_1(q).$ 
\end{proof}
\end{thm}

\begin{rmk}
\label{openprob}
We are not able to prove that there exist no solutions when $(\alpha,\beta)\in(\tilde{s}_+,+\infty)\times\{\lambda_1(q)\}$: this is mainly due to the lack of global $C^{1,\tau}$ regularity of solutions to \eqref{prob}. Indeed, in \cite[Proposition 4(ii)]{BT}, regularity allows to use $\frac{\phi_q^p}{u^{p-1}}$ as test function in \GEV{$\lambda_1(q)+s,\lambda_1(q)$}, being $s>s^*_+$. It is worth noticing that using either $\frac{\phi_{q,n}^p}{u^{p-1}}$ or $\frac{\phi_{q,n}^p}{(u+\delta)^{p-1}}$ (see Proposition \ref{lambdaprop}(v) and Lemma \ref{Picone2}, respectively) as test function gives rise to technical issues.
\end{rmk}

\begin{rmk}
\label{pqlaplacian}
We highlight that, whenever \eqref{LI} holds true and $\inf_\Omega a > 0$ (in particular for the $(p,q)$-Laplacian), non-existence of positive solutions when $(\alpha,\beta)=(\tilde{s}_+,\lambda_1(q))$ comes from the Picone identities in \cite{A,I}. Indeed, suppose by contradiction that there exists $u\in W^{1,\theta}_0(\Omega)$ positive solution to \GEV{$\tilde{s}_+,\lambda_1(q)$}.

Notice that \eqref{LI} forces $u\neq k\phi_q$ for all $k\in\R$: otherwise,
$$ -k^{p-1}\Delta_p^a \phi_q - k^{q-1}\Delta_q \phi_q = \tilde{s}_+ a k^{p-1} \phi_q^{p-1} + \lambda_1(q)k^{q-1}\phi_q^{q-1}, $$
whence
$$ -\Delta_p^a \phi_q = \tilde{s}_+ a \phi_q^{p-1}, $$
which implies that $\phi_q$ is a positive eigenfunction of $(-\Delta_p^a,W^{1,\theta_0}(\Omega))$, forcing $\tilde{s}_+=\lambda_1^a(p)$, in contrast to \eqref{LI}.

Hence, testing \GEV{$\tilde{s}_+,\lambda_1(q)$} with $\frac{\phi_q^p}{u^{p-1}}$, which is allowed by nonlinear regularity theory \cite{LI}, and using the Picone identities \cite{A,I}, we deduce
\begin{equation*}
\begin{split}
&\int_{\Omega} a|\nabla \phi_q|^{p} \dx +\lambda_1(q) \int_{\Omega} u^{q-p} \phi_q^p \dx \\
&>\int_{\Omega} a|\nabla u|^{p-2} \nabla u \cdot \nabla \left(\frac{\phi_q^p}{u^{p-1}}\right)\dx +\int_{\Omega} |\nabla u|^{q-2} \nabla u \cdot \nabla \left(\frac{\phi_q^p}{u^{p-1}}\right)\dx\\
&=\tilde{s}_+ \int_{\Omega}a\phi_q^p\dx +\lambda_1(q)\int_{\Omega} u^{q-p} \phi_q^p \dx,
\end{split}
\end{equation*}
contradicting the definition of $\tilde{s}_+$.
\end{rmk}



\section*{Acknowledgments}
\noindent
The authors warmly thank Lorenzo Brasco and Sunra Mosconi for the fruitful conversations about Picone-type inequalities. They are also indebted to Vladimir Bobkov for his valuable comments. \\
The authors are member of the {\em Gruppo Nazionale per l'Analisi Ma\-te\-ma\-ti\-ca, la Probabilit\`a e le loro Applicazioni}
(GNAMPA) of the {\em Istituto Nazionale di Alta Matematica} (INdAM); they are partially supported by the INdAM-GNAMPA Project 2023 titled {\em Problemi ellittici e parabolici con termini di reazione singolari e convettivi} (E53C22001930001). \\
Umberto Guarnotta is supported also by `PERITO' PRA 2020--2022 `PIACERI' Linea 3 of the University of Catania. \\
This study was partly funded by: Research project of MIUR (Italian Ministry of Education, University and Research) PRIN 2022 ``Nonlinear differential problems with applications to real phenomena'' (Grant Number: 2022ZXZTN2).\\
\\
\textbf{Conflict of interest statement:} On behalf of all authors, the corresponding author states that there is no conflict of interest.
\\
\textbf{Author contributions:} The authors have accepted responsibility for the entire content of this manuscript and
approved its submission. The authors have equal contribution for this paper, from the methodology to the writing, revision and editing.
\begin{small}

	\end{small}

\end{document}